\documentclass{article}

\usepackage{amsfonts}
\usepackage{amsmath}
\usepackage{a4wide}
\usepackage{amssymb}

\newtheorem{lem}{Lemma}[section]
\newtheorem{theo}[lem]{Theorem}
\newtheorem{coro}[lem]{Corollary}
\newtheorem{propo}[lem]{Proposition}
\newtheorem{rema}[lem]{Remark}
\newtheorem{defi}[lem]{Definition}

\newenvironment{lemma}
{\begin{lem}\sl } {\end{lem}}

\newenvironment{theorem}
{\begin{theo}\sl } {\end{theo}}

\newenvironment{corollary}
{\begin{coro}\sl } {\end{coro}}

\newenvironment{proposition}
{\begin{propo}\sl } {\end{propo}}

\newenvironment{proof}{\paragraph*{Proof}}
{\par}

\newcommand\qed{\hfill$\square$}

\newcommand\CC{{\mathcal C}}
\newcommand\OO{{\mathcal O}}

\newcommand\calR{{\mathcal R}}

\newcommand\gal{{\mathrm{Gal}}}
\newcommand\ord{{\mathrm{ord}}}
\newcommand\GL{{\mathrm{GL}}}
\newcommand\SL{{\mathrm{SL}}}

\newcommand\Mu{{\mathrm M}}
\newcommand\eps\varepsilon
\newcommand\ph\varphi
\newcommand\C{{\mathbb C}}
\newcommand\F{{\mathbb F}}
\newcommand\Q{{\mathbb Q}}
\newcommand\PPP{{\mathbb P}}
\newcommand\Z{{\mathbb Z}}
\newcommand\height{{\mathrm h}}
\newcommand\im{{\mathrm {Im}\,}}
\newcommand\sign{{\mathrm {sign}\,}}
\newcommand\HH{{\mathcal H}}
\newcommand\bfa{{\mathbf a}}
\newcommand\bfb{{\mathbf b}}

\newcommand\bfA{{\mathbf A}}

\newcommand\tilbfA{{\widetilde \bfA}}
\newcommand\tilD{{\widetilde D}}

\newcommand\topbot[2]{{\genfrac{}{}{0pt}{}{{#1}}{{#2}}}}

\newcommand\spl{{\mathrm{split}}}

\title{Bounds for Integral $j$-Invariants and Cartan Structures on Elliptic Curves}
 
\author{Yuri Bilu, Pierre Parent (Universit\'e de Bordeaux~I)}

1

\begin{document}

\maketitle

\begin{abstract}
We bound  the $j$-invariant of integral points 
on a modular curve in terms of the congruence group defining the curve. 
We apply this to prove that the modular curve $X_{\mathrm{
split}} (p^3 )$ has no non-trivial rational point if~$p$  is a sufficiently large prime number. Assuming the GRH, one can replace~$p^3$ by~$p^2$.

\vspace{2cm}

AMS 2000 Mathematics Subject Classification  11G18 (primary), 11G05, 11G16 (secondary). 
\end{abstract}

\section{Introduction}

Let ${N\ge 2}$ be an integer and $X(N)$ the principal modular curve of level~$N$. Further, 
let~$G$ a subgroup of $\GL_{2} (\Z /N\Z )$  and let $X_G$ be the corresponding modular 
curve. This curve is defined over $\Q  (\zeta_{N})^{\det (G)}$, so in particular it is 
defined over $\Q$ if ${\det (G)=(\Z /N\Z )^\times}$. (Through all this paper, we say that 
an algebraic curve is  \textsl{defined} over a field if it has a geometrically integral model over this 
field.) As usual, we denote by~$Y_G$ the finite 
part of~$X_G$ (that is,~$X_G$ deprived  of the cusps). If~$X_G$ is defined over a 
number field~$K$, the curve~$X_{G}$ has a natural (modular) model over ${\OO=\OO_K}$ 
that we still denote by~$X_{G}$. The cusps define a closed subscheme of~$X_{G}$ 
over~$\OO$, and we define the relative curve~$Y_{G}$ over~$\OO$ as~$X_{G}$ deprived 
of the cusps. If~$S$ is a finite set of places of~$K$ containing the infinite places, then the 
set of $S$-integral points $Y_G(\OO_S)$ consists of those ${P\in Y_G(K)}$ for which 
${j(P)\in \OO_S}$, where~$j$ is, as usual, the modular invariant and ${\OO_S=\OO_{K,S}}$
is the ring of~$S$-integers. 

In its simplest form, the first principal result of this article gives an explicit upper bound for
${j(P)\in \Z}$ under certain Galois condition on the cusps. More precisely, we prove the 
following.

\begin{theorem}
\label{th1}
Assume that~$X_G$ is defined over~$\Q$, and assume that the absolute Galois group 
$\gal(\bar\Q/\Q)$ acts non-transitively on the cusps of~$X_G$. Then for any ${P\in Y_G
(\Z)}$ we have
\begin{equation}
\label{eth1}
\log|j(P)|\le 30|G|N^2\log N. 
\end{equation}
\end{theorem}

This result was announced in~\cite{BP08}. Because of an inaccuracy in the proof given 
in~\cite{BP08}, the $\log N$ factor is missing therein (which, however, does not affect 
the arithmetical applications of this theorem). 

Actually, we obtain two versions of Theorem~\ref{th1}. One (Theorem~\ref{tbo} below) 
is quite general, applies to any number field and a ring of~$S$-integers in it, but the bound is 
slightly weaker. The other (see Section~\ref{sspec}) is less general than 
Theorem~\ref{th1}, and applies only to certain particular groups~$G$ (the normalizers of 
split tori), but the bound is sharper. 

To state Theorem~\ref{tbo}, we need to  introduce some 
notation. We denote by ${\height(\cdot)}$ the usual absolute logarithmic height (see Subsection~\ref{ssnota}).  For ${P\in X_G(\bar \Q)}$ we shall write 
${\height(P)=\height\bigl(j(P)\bigr)}$.
For a number field~$K$ we denote by ${\CC=\CC(G)}$ the set of cusps of~$X_G$, and by $\CC(G,K)$ the 
set of $\gal(\bar K/K)$-orbits of~$\CC$.

\begin{theorem}
\label{tbo}
Let~$K$ be a number field and~$S$ a finite set of places of~$K$ (including all the 
infinite places). Let~$G$ be a subgroup of $\GL_2(\Z/N\Z)$ such 
that~$X_G$ is defined over~$K$.  Assume that 
${|\CC(G,K)|> |S|}$
(the ``Runge condition''). Then for any ${P\in Y_G(\OO_S)}$ we have 
\begin{equation}
\label{etbo}
\height(P) \le s^{s/2+1}\left(|G|N^2\right)^sN(\calR+30), 
\end{equation}
where ${s=|S|}$ and
\begin{equation}
\label{ecalr}
\calR=\calR(N,S) =\sum_{\topbot{p\mid N}{v\mid p\ \text{for some}\ v\in S}}
\frac{\log p}{p-1},
\end{equation}
the sum being over all the prime divisors of~$N$ below the (finite) places from~$S$ (in particular ${\calR=0}$ if~$S$ consists only of infinite places). 
\end{theorem}

While this theorem applies in the set-up of Theorem~\ref{th1}, it implies a slightly weaker 
result, with ${30|G|N^3}$  on the right. Mention also that the constant~$30$ is not best possible for the method and can be easily reduced, at least for 
large~$N$.

These theorems are proved  in Sections~\ref{sproof} and~\ref{sth1} by a variation of the 
method of Runge, after some preparation in Sections~\ref{sest},~\ref{snecu} 
and~\ref{smuni}.  In Section~\ref{sspec} we obtain an especially 
sharp version of Theorem~\ref{th1} for the case of split tori. For a general discussion of 
Runge's method see~\cite{Bo83,Le08}.

In Section~\ref{sspli}  we apply these results to the arithmetic of elliptic curves. 
We are motivated by a question of Serre, who proved~\cite{Se72} that for any elliptic 
curve~$E$ without complex multiplication (CM in the sequel), there exists  a constant 
$p_0(E)$ such that for every prime ${p>p_0(E)}$ the natural Galois representation 
${\rho_{E,p}:\gal
(\bar{\Q} /\Q )\to \GL (E[p])\cong\GL_2(\F_p)}$ 
is surjective. Masser and W\"ustholz~\cite{MW93} gave an effective version of Serre's result; see also the more recent work of Cojocaru and Hall \cite{Co05,CH05}.

Serre asks whether~$p_0$ can be made independent of~$E$:

\begin{quotation}{\sl
\noindent
does there exist an absolute constant~$p_0$ such that for any non-CM elliptic curve~$E$ 
over~$\Q$ and any prime ${p>p_0}$ the  Galois representation~$\rho_{E,p}$ is 
surjective? }
\end{quotation}
The general guess is that $p_{0}=37$ would probably do.  

We obtain several results on Serre's question. One knows that, for a positive answer, it is
sufficient to bound the primes~$p$ such that a non-CM curve may have a Galois structure 
included in the normalizer of a (split or nonsplit) Cartan subgroup of $\GL_2(\F_p)$. 
Equivalently, one would like to prove that, for large~$p$, the only rational points of the 
modular curves $X_{\mathrm{split}} (p)$ and $X_{\mathrm{nonsplit}} (p)$ are the cusps 
and CM points, in which case we will say that the rational points are \textsl{trivial} (for 
the precise definition of these curves see Section~\ref{sspec}). In \cite{Pa05,Re08} it 
was proved, by very different techniques, that $X_{\mathrm{split}}(p)(\Q)$ is trivial for a 
(large) positive density of primes; but the methods of loc. cit. have failed to prevent a 
complementary set of primes from escaping them. Here we consider 
Cartan structures modulo some higher power of a prime, and we prove the following.

\begin{theorem}
\label{tsp}
For large enough prime $p$, every point in $X_{\mathrm{split}} (p^3 )(\Q )$ is either 
a CM-point or a cusp. 
Assuming the Generalized Riemann Hypothesis for the zeta functions of number fields (GRH 
in the sequel), the same holds true for $X_{\mathrm{split}} (p^2 )(\Q )$.
\end{theorem}
Equivalently, for large
 enough~$p$ and for any non-CM elliptic curve~$E$ defined over~$\Q$, the image of the 
 Galois representation $\rho_{E,p^3}$ is not contained in the normalizer of a split Cartan 
 subgroup of $\GL_2(\Z/p^3\Z)$ (and $p^3$ can be replaced by~$p^2$ assuming~GRH).

The second part, where one assumes GRH, was sketched in~\cite{BP08}, with 
level $p^5$. Here we manage to reduce to level $p^2$ thanks to the refined bound from 
Section~\ref{sspec}, and we prove the first, unconditional assertion of Theorem~\ref{tsp}
by applying the isogeny estimate of Masser and W\"ustholz~\cite{MW90}, made explicit 
by Pellarin~\cite{Pe01}.

Call a prime number~$p$ \textsl{deficient} for an elliptic curve~$E$ if $\rho_{E,p}$ is 
not surjective; we call~$p$ a \textsl{(non-)split Cartan deficient prime} if the image of 
$\rho_{E,p}$ is contained in the normalizer of a (non) split Cartan subgroup. Following a suggestion of L.~Merel and J.~Oesterl\'e, we  prove that the split 
Cartan deficiencies are bounded, with at most~$2$ exceptions. 

\begin{theorem}
\label{ttwo}
There exists an absolute effective constant~$p_0$ such that for any non-CM elliptic curve 
$E/\Q$, all but~$2$ split Cartan deficient primes do not exceed~$p_0$.
\end{theorem}

\paragraph{Acknowledgments}
We thank Daniel Bertrand, Henri Cohen, Lo\"ic Merel, Joseph Oesterl\'e, Vinayak Vatsal and Yuri Zarhin for stimulating discussions and useful suggestions. We specially acknowledge that the idea of Theorem~\ref{ttwo} came out in a conversation with Merel and Oesterl\'e.

\subsection{Notation, conventions}
\label{ssnota}
Everywhere in this article  $\log$ and~$\arg$ stand for the principal branches of the 
complex logarithm and argument functions; that is, for any ${z\in \C^\times}$ we have
${-\pi<\im\log z=\arg z\le \pi}$.
We shall systematically use, often without special reference, the estimates of the kind
\begin{equation}
\label{eschwa}
\begin{aligned}
|\log (1+z)| &\le \frac{|\log(1-r)|}r|z|, \qquad
\left|e^z-1\right|\le \frac{e^r-1}r|z|, \\
\left|(1+z)^A-1-Az\right|&\le \frac{\left|1+\eps r)^A-1-\eps Ar\right|}{r^2}|z|^2 \quad (\eps=\sign A), 
\end{aligned} 
\end{equation}
etc., for ${|z|\le r<1}$. They can be easily deduced from the Schwarz lemma.

Let~$\HH$ denote the upper half-plane of the complex plane:
${\HH=\{\tau\in \C : \im\tau>0\}}$. For ${\tau\in \HH}$ we put ${q_\tau=e^{2\pi i \tau}}$. 
We put ${\bar\HH=\HH\cup\Q\cup\{i\infty\}}$. If~$\Gamma$ is a pull-back of ${G\cap
\SL_2(\Z/N\Z)}$ to $\SL_2(\Z)$, then the set $X_G(\C)$ of complex points is analytically
isomorphic to the quotient  ${X_\Gamma=\bar\HH/\Gamma}$, supplied with the properly
defined topology and analytic structure \cite{La76,Sh71}.  

We denote by~$D$ the standard fundamental 
domain of $\SL_2(\Z)$ (the hyperbolic triangle with vertices $e^{\pi i/3}$, 
$e^{2\pi i/3}$ and $i\infty$, together with the geodesic segments ${[i,e^{2\pi i/3}]}$ 
and ${[e^{2\pi i/3},i\infty]}$).
Notice that for ${\tau\in D}$ we have ${|q_\tau|\le e^{-\pi\sqrt3}<0.005}$, 
which will be systematically used without special reference.

For ${\bfa=(a_1,a_2)\in \Q^2}$ we put 
${\ell_\bfa=B_2\bigl(a_1-\lfloor a_1\rfloor\bigr)}$
where ${B_2(T)=T^2-T+1/6}$ is the second Bernoulli polynomial. The quantity~$\ell_\bfa$ is $\Z^2$-periodic in~$\bfa$ and is thereby well-defined for ${\bfa\in (\Q/\Z)^2}$ as well: for such~$\bfa$ we have ${\ell_\bfa=B_2(\widetilde{a}_1)}$, where~$\widetilde{a}_1$ is the lifting of the first coordinate of~$\bfa$ to  the interval $[0,1)$. Obviously, 
${|\ell_\bfa|\le 1/12}$;
this will  also be often used without special reference. 

We fix, once and for all, an algebraic closure~$\bar\Q$ of~$\Q$, which is assumed to be a 
subfield of~$\C$. In particular, for every ${a\in\Q}$ we have the well defined root of unity 
${e(a)=e^{2\pi i a}\in \bar\Q}$.
Every number field used in this article is presumed to be  contained in the fixed~$\bar\Q$.
If~$K$ is such a number field and~$v$ is a valuation on~$K$, then we tacitly assume
than~$v$ is somehow extended to ${\bar\Q=\bar K}$; equivalently, we fix an algebraic
closure~$\bar K_v$ and an embedding ${\bar\Q\hookrightarrow\bar K_v}$. In
particular, the roots of unity $e(a)$ are well-defined elements of~$\bar K_v$. 

For a number field~$K$ we denote by~$M_K$ the set of all valuations (or places) of~$K$ normalized to extend the usual infinite and $p$-adic valuations of~$\Q$: ${|2|_v=2}$ if ${v\in M_K}$ is infinite, and ${|p|_v=p^{-1}}$ if~$v$ extends the $p$-adic valuation of~$\Q$. In the finite case we sometimes use the additive notation $v(\cdot)$, normalized to have ${v(p)=1}$. We denote by~$M_K^\infty$ and~$M_K^0$ the subsets of~$M_K$ consisting of the infinite (archimedean) and the finite (non-archimedean) valuations, respectively. 

Recall the definition of the absolute logarithmic height $\height(\cdot)$. For 
${\alpha \in \bar\Q}$ we pick a number field~$K$ containing~$\alpha$ and put 
${\height(\alpha) = [K:\Q]^{-1}\sum_{v\in M_K}[K_v:\Q_v]\log^+|\alpha|_v}$,
where the valuations on~$K$ are normalized to extend standard infinite and $p$-adic 
valuations on~$\Q$ and ${\log^+x=\log\max\{x,1\}}$.  The value of $\height(\alpha)$  is known to 
be independent on the particular choice of~$K$. As usual, we extend the definition of the 
height to ${\PPP^1(\bar \Q)=\bar\Q\cup\{\infty\}}$ by setting ${\height (\infty)=
0}$. If~$\alpha$ is a rational integer or an imaginary quadratic integer then ${\height
(\alpha)=\log|\alpha|}$.

\section{Estimates for  Modular Functions at Infinity}
\label{sest}

The results of this section must be known, but we did not find them in the available 
literature, so we state and prove them here. Most of the results of this section are stronger 
than what we actually need, but we prefer to state them in this sharp form for the sake of 
further applications.  

\subsection{Estimating the $j$-Function}

Recall that the modular 
$j$-invariant ${j:\HH\to\C}$ is defined by ${j(\tau)=(12c_2(\tau))^3/\Delta(\tau)}$, 
where 
$$
c_2(\tau)= \frac{(2\pi i)^4}{12}\left(1+240\sum_{n=1}^\infty
\frac{n^3q_\tau^n}{1-q^n}\right)
$$
(see, for instance, \cite[Section~4.2]{La73}) and
${\Delta(\tau) =(2\pi i)^{12}q\prod_{n=1}^\infty(1-q^n)^{24}}$ (here ${q=q_\tau=e^{2\pi i\tau}}$). Also,~$j$
has the familiar Fourier expansion
${j(\tau)= q^{-1}+ 744+196884q+\ldots}$.

\begin{proposition}
\label{pqj}
For ${\tau\in \HH}$ such that ${|q_\tau|\le 0.005}$ (and, in particular, for every
${\tau\in D}$) we have 
\begin{equation}
\label{ejq}
\left|j(\tau)-q_\tau^{-1}-744\right| \le 330000|q_\tau|.
\end{equation}
\end{proposition}
(Recall that~$D$ is the standard fundamental domain for $\SL_2(\Z)$.)

\begin{proof}
We  write ${q=q_\tau}$.  Using the estimate ${n^3\le 3^n}$ for ${n\ge 3}$, we find that 
for ${|q|<1/3}$ 
\begin{align*}
\left|\frac{12}{(2\pi i)^4}c_2(\tau)-1-240q\right| &\le 240\left(
\frac{|q|^2}{1-|q|}+ \sum_{n=2}^\infty\frac{n^3|q|^n}{1-|q|^n}\right)\\
&\le \frac{240}{1-|q|}\left(|q|^2+8|q|^2+\sum_{n=3}^\infty|3q|^n\right)\\
&=\frac{2160}{(1-|q|)(1-3|q|)}|q|^2,
\end{align*}
and for ${|q|\le 0.005}$ we obtain 
\begin{equation}
\label{ectwo}
\left|\frac{12}{(2\pi i)^4}c_2(\tau)-1-240q\right|\le 2204|q|^2.
\end{equation}
Further, using~(\ref{eschwa}), we obtain,  for ${|q|\le0.005}$, 
$$
\left|\log\frac{(2\pi i)^{12}q(1-q)^{24}}{\Delta(\tau)}\right|=
24\left|\sum_{n=2}^\infty\log\left(1-q^n\right)\right| \le 24.1\sum_{n=2}^\infty |q|^n \le 24.3|q|^2.
$$
Hence
\begin{align*}
\left|\frac{(2\pi i)^{12}q}{\Delta(\tau)}-1-24q\right|&\le 
\left|(1-q)^{-24}\right|\left|\frac{(2\pi i)^{12}q(1-q)^{24}}{\Delta(\tau)}-1\right|+\left|(1-q)^{-24}-1-24q\right|\\
&\le 1.13\left|\log\frac{(2\pi i)^{12}q(1-q)^{24}}{\Delta(\tau)}\right|+314|q|^2\le 342|q|^2.
\end{align*}
Combining this with~(\ref{ectwo}), we obtain~(\ref{ejq}) after a tiresome, but straightforward calculation.\qed
\end{proof}

\medskip

\begin{corollary}
\label{cdplus}
For any ${\tau\in D}$ we have 
either ${|j(\tau)|\le 2500}$ or ${|q_\tau|< 0.001}$.\qed
\end{corollary}

\subsection{Estimating Siegel's Functions}
\label{sssieg}
For a rational number~$a$ we define ${q^a=e^{2\pi i a\tau}}$. 
Let ${\bfa=(a_1,a_2)\in \Q^2}$ be such that ${\bfa\notin \Z^2}$, and let ${g_\bfa:\HH\to \C}$ be the corresponding \emph{Siegel function} \cite[Section~2.1]{KL81}. Then, putting ${z=a_1\tau+a_2}$ and ${q_z=q_\tau^{a_1}e(a_2)}$, where ${e(a)=e^{2\pi i a}}$,  we have the following infinite product presentation for~$g_\bfa$ \cite[page~29]{KL81} (where $B_2(T)$ is the second Bernoulli polynomial):
\begin{equation}
\label{epga}
g_\bfa(\tau)= -q_\tau^{B_2(a_1)/2}e\left(\frac{a_2(a_1-1)}2\right)(1-q_z)\prod_{n=1}^\infty\left(1-q_\tau^nq_z\right)\left(1-q_\tau^n/q_z\right).
\end{equation}
We also have \cite[pages 29--30]{KL81} the relations
\begin{align}
\label{eperga}
g_\bfa\circ\gamma &=g_{\bfa\gamma} \cdot(\text{a root of unity}) \quad \text{for} \quad \gamma\in\Gamma(1),\\
\label{eaa'}
g_\bfa&=g_{\bfa'}\cdot(\text{a root of unity})
 \quad \text{when} \quad \bfa\equiv\bfa'\mod\Z^2.
\end{align}
Remark that the roots of unity  in~(\ref{eperga}) and~(\ref{eaa'}) are of order dividing $12N$; this will be used later.

The order of vanishing of~$g_\bfa$ at $i\infty$ (that is, the only rational number~$\ell$ such that the limit 
${{\displaystyle\lim_{\tau\to i\infty}}q_\tau^{-\ell}g_\bfa(\tau)}$
exists and is non-zero) is equal to the number~$\ell_\bfa$, defined in Subsection~\ref{ssnota},  
see \cite[page~31]{KL81}. 

\begin{proposition}
\label{psiar}
Assume that ${a_1\notin\Z}$ and 
let~$N$ be a denominator of~$a_1$  (that is, a positive integer satisfying ${Na_1\in \Z}$). 
Then for  ${|q_\tau|\le 10^{-N}}$ we have
\begin{equation}
\label{ega1}
\Bigl|\log \left|g_\bfa(\tau)\right|- \ell_\bfa\log|q_\tau|\Bigr| \le 3|q_\tau|^{1/N}.
\end{equation}
Further, assume that ${a_1\in \Z}$. Then for ${|q_\tau|\le 0.1}$ we have 
\begin{equation}
\label{ega0}
\Bigl|\log \left|g_\bfa(\tau)\right|- \ell_\bfa\log|q_\tau|-\log\left|1-e(a_2)\right|\Bigr| \le 3|q_\tau|.
\end{equation}
\end{proposition}

\begin{proof}
Due to~(\ref{eperga}), we may assume that ${0\le a_1<1}$ and distinguish between the cases ${0<a_1<1}$ and ${a_1=0}$. 
Assume first that ${0<a_1<1}$. According to~(\ref{epga}), the left-hand side of~(\ref{ega1}) is equal to 
$$
\left|\log|1-q_z| + \log|1-q_\tau/q_z|+\sum_{n=1}^\infty\log |1-q_\tau^n q_z|+\sum_{n=2}^\infty\log |1-q_\tau^n/q_z|\right|.
$$
Since ${0<a_1<1}$, both~$|q_z|$ and~$|q_\tau/q_z|$ are bounded  by $|q_\tau|^{1/N}$, which does not exceed~$0.1$ because ${|q_\tau|\le 10^{-N}}$. (One can say even more: one of these numbers is bounded by $|q_\tau|^{1/N}$ and the other by $|q_\tau|^{1/2}$, which will be used later.)
  Hence, using~(\ref{eschwa}) with ${r=0.1}$, we obtain 
\begin{equation}
\label{elqz}
\Bigl|\log|1-q_z| + \log|1-q_\tau/q_z|\Bigr|\le 2\cdot1.1|q_\tau|^{1/N}.
\end{equation}
Similarly, each of $\left|q_\tau^n q_z\right|$ and $\left|q_\tau^n/ q_z\right|$  does not exceed $0.1$, whence
\begin{equation}
\label{esumn}
\left|\sum_{n=1}^\infty\log |1-q_\tau^n q_z|+\sum_{n=2}^\infty\log |1-q_\tau^n/q_z|\right|\le 1.1\frac{\left|q_\tau q_z\right|+\left|q_\tau^2/q_z\right|}{1-|q_\tau|}\le 3|q_\tau|\cdot|q_\tau|^{1/N} ,
\end{equation}
which does not exceed $0.3|q_\tau|^{1/N}$. This proves~(\ref{ega1}). 

When ${a_1=0}$ then ${a_2\notin\Z}$ and ${\zeta =e(a_2)\ne 1}$. Further, we have 
${q_z=\zeta}$ and the left-hand side of~(\ref{ega1}) is 
${\bigl|\sum_{n=1}^\infty\log|1-q_\tau^n\zeta|+ \sum_{n=1}^\infty
\log|1-q_\tau^n/\zeta|\bigr|}$.
Estimating the sums using~(\ref{eschwa}), we obtain~(\ref{ega0}). \qed 
\end{proof}

\medskip

Since Siegel's functions has no poles nor zeros on the upper half plane~$\HH$, it should be 
bounded from above and from below on any compact subset of~$\HH$. In particular, it 
should be bounded where~$j$ is bounded. Here is a quantitative version of this. 

\begin{proposition}
\label{pgalar}
Let ${\bfa \in \Q^2}$ be of order ${N>1}$ in ${(\Q/\Z)^2}$. Then for any ${\tau\in 
\HH}$ we have     
\begin{equation}
\label{esmallj}
\bigl|\log |g_\bfa(\tau)|\bigr|\le \frac1{12}\log\bigl(|j(\tau)|+2200\bigr)+\log 
N +0.1.
\end{equation}
\end{proposition}

\begin{proof}
Replacing~$\tau$ by $\gamma\tau$ and~$g_\bfa$ by~$g_{\bfa\gamma^{-1}}$ with a 
suitable ${\gamma\in \Gamma(1)}$, we may assume that ${\tau\in D}$, and in 
particular ${|q_\tau|<e^{-\pi\sqrt3}}$. We may also assume that ${0\le a_1<1}$. 
Now we argue as in the previous proof, the only difference being that in the case  ${0<a_1
<1}$, we replace~(\ref{elqz}) by
\begin{equation}
\label{eqqzt}
\bigl|\log |1-q_z|+\log|1-q_\tau/q_z|\bigr|\le  \left|\log\bigl|1-e^{-\pi\sqrt3/N}
\bigr|\right|+ \left|\log\bigl|1-e^{-\pi\sqrt3/2}\bigr|\right|
\le \log N + 0.07 .
\end{equation}
Indeed, among the numbers $|q_z|$ and $|q_\tau/q_z|$ one is bounded by $|q_\tau|^{1/N}$ 
and the other is bounded by $|q_\tau|^{1/2}$, which implies~(\ref{eqqzt}). 

Estimate~(\ref{esumn}) holds again, the right-hand side being bounded by ${3|q_\tau|\le 
0.02}$. 
Thus, in the case ${0<a_1<1}$ we have
\begin{equation}
\label{enzer}
\bigl|\log \left|g_\bfa(\tau)\right|- \ell_\bfa\log|q_\tau|\bigr| \le \log N+0.1.
\end{equation}
In the case ${a_1=0}$ we can use estimate~(\ref{ega0}), which implies 
${\bigl|\log \left|g_\bfa(\tau)\right|- \ell_\bfa\log|q_\tau|\bigr| \le \log 2+
0.02}$,
and \textit{a fortiori}~(\ref{enzer}).

Finally,
Proposition~\ref{pqj} implies that 
${\bigl|\log |q_\tau|\bigr|\le \log\bigl(|j(\tau)|+2200\bigr)}$.
Combining this~(\ref{enzer}), and using the inequality ${|\ell_\bfa|\le 1/12}$,  we obtain~(\ref{esmallj}). \qed 
\end{proof}

\subsection{Non-archimedean versions}

We also need non-archimedean versions of some of the above inequalities. In this 
subsection~$K_v$ is a field complete with respect to a non-archimedean valuation~$v$ and~$\bar K_v$ its algebraic closure. Let ${q\in K_v}$ satisfy ${|q|_v<1}$. Put ${j(q)= q^{-1}+ 744+196884q+\ldots}$ Further,
for  ${\bfa=(a_1,a_2)\in \Q^2}$ such that ${\bfa\notin\Z^2}$   put ${q_z=q^{a_1}e(a_2)}$ and define 
$$
g_\bfa=g_\bfa(q)= -q^{B_2(a_1)/2}e\left(\frac{a_2(a_1-1)}2\right)(1-q_z)\prod_{n=1}^\infty\left(1-q^nq_z\right)\left(1-q^n/q_z\right).
$$
This expression is not well-defined because we use rational powers of~$q$. However, if we 
fix ${q^{1/2N^2}\in\bar K_v}$, where~$N$ is the order of~$\bfa$ in $(\Q/\Z)^2$, then 
everything becomes well-defined, and, moreover, we again have~(\ref{eperga}) 
and~(\ref{eaa'}). The statement of the following proposition is independent on the 
particular choice of~$q^{1/2N^2}$.  Recall that ${\ell_\bfa=B_2\bigl(a_1-\lfloor a_1\rfloor\bigr)/2}$.

\begin{proposition}
\label{psinar}
In the above set-up, when ${a_1\notin \Z}$ we have
${\log\left|g_\bfa(q)\right|_v=\ell_\bfa\log|q|_v}$, 
and when ${a_1\in \Z}$ we have 
${\log\left|g_\bfa(q)\right|_v=\ell_\bfa\log|q|_v + \log|1-e(a_2)|_v}$. 

\end{proposition}

\begin{proof}
This is obvious when ${0\le a_1<1}$, and the general case reduces to this one using~(\ref{eaa'}). \qed
\end{proof}

\begin{corollary}
\label{cnaglob}
In the above set-up, we have 
$$
\left|\log\left|g_\bfa(q)\right|_v\right|\le \frac1{12}\log|j(q)|_v+
\begin{cases}0&\text {if ${v(N)=0}$}\\
\frac{\log p}{p-1}& \text{if ${v(N)>0}$ and~$p$ is the prime below~$v$}.
\end{cases}
$$
\end{corollary}

\begin{proof}
It suffices to notice that  ${|j(q)|_v=|q|_v^{-1}}$, that ${|\ell_\bfa|\le 1/12}$, that 
${|1-e(a_2)|_v=1}$ if ${v(N)=0}$, and that
${1\ge|1-e(a_2)|_v\ge p^{-1/(p-1)}}$ if ${v\mid p\mid N}$. \qed
\end{proof}

\section{Locating the ``nearest cusp''}
\label{snecu}
Let~$N$ be a positive integer,~$G$ a subgroup of ${\GL_2(\Z/N\Z)}$ and~$X_G$ the 
corresponding modular curve, defined over a number field~$K$. In this section we fix a 
valuation~$v$ of~$K$. We denote by~$\OO$ the ring of integers of~$K$ and by~$K_v$ the 
$v$-completion of~$K$. When~$v$ is non-archimedean, we denote by~$\OO_{v}$ the ring of integers of~$K_{v}$, and by~$k_{v}$ its
residue field at~$v$. As usual~$\zeta_{N}$ will denote a primitive~$N$-th root of unity. 
Recall that when we say that a curve ``is defined over'' a field, it means that this curve has a 
geometrically integral model over that field.

Let~$P$ be a point on $X_G(K_v)$ such that $|j(P)|$ is ``large''. Then it is intuitively clear 
that, in the $v$-adic metric,~$P$ is situated ``near'' a cusp of~$X_G$. The purpose of this 
section is to make this intuitive observation precise and explicit. We shall locate this 
``nearest'' cusp and specify what the word ``near'' means.

We first recall the following description of the cuspidal locus of~$X(N)$ (for more details see 
e.g.~\cite[Chapitres V and VII)]{DR73}). The cusps of~$X(N)$ define a closed subscheme
of the smooth locus of the modular model of~$X(N)$ over~$\Z [\zeta_{N}]$. Fix a 
uniformization~$X(N)(\C )\simeq\bar{\HH} /\Gamma (N)$, let~$c_\infty$ be the 
cusp corresponding to~$\infty\in\bar{\HH}$, and write ${q^{1/N} =
e^{2i\pi\tau/N}}$ the usual parameter. If~$c=\gamma (c_\infty )$, for some~$\gamma\in
\SL_{2} (\Z )$, is another cusp, denote by~$q_c :=q\circ \gamma^{-1}$ the parameter 
on~$X(N)(\C )$ at $c$. It follows from~\cite[Chapitre VII, Corollaire 2.5]{DR73} that the 
completion of the curve $X(N)$ over $\Z [\zeta_{N}]$ along the section $c$ is isomorphic 
to ${{\mathrm{Spec}}\bigl(\Z [\zeta_{N}]\,
[[q_c^{1/N} ]]\bigr)}$. In other words, the parameter $q_c^{1/N}$ at~$c$ on $X(N)(\C )$ is 
actually defined over~$\Z [\zeta_{N}]$, that is~$q_c^{1/N}$ comes from an element of the 
completed local ring~$\hat{\cal O}_{X(N),c}$ of the modular model of~$X(N)$ over $\Z
[\zeta_{N}]$, along the section~$c$. Moreover the modular interpretation associates with 
each cusp a N\'eron polygon~$C$ with~$N$ sides on~$\Z [\zeta_N ]$, endowed with its 
structure of generalized elliptic curve, and enhanced with a basis of~$C[N]\simeq 
\Z /N\Z \times\mu_N =\langle q^{1/N} ,\zeta_N \rangle$ such that the determinant
of this basis is~$1$, and two bases are identified if they are conjugate by the subgroup ${\pm U=\pm\left(\topbot10\topbot *1\right)}$ of $\GL_2(\Z/N\Z)$, the
action being ${\left(\topbot\epsilon0\topbot a\epsilon\right):(q^{1/N},\zeta_N )\mapsto (q^{\epsilon /N}\zeta_N^a ,\zeta_N^{
\epsilon})}$, where ${\epsilon =\pm 1}$ and ${a\in \Z /N\Z}$. We may, for instance, interpret 
$c_{\infty}$ as the orbit ${\left\{ (C,(q^{\epsilon/N}\zeta_N^a ,\zeta_{N}^{\epsilon}
)),\ \epsilon\in\{\pm 1\} ,\ a\in\Z /N\Z\right\}}$ of enhanced N\'eron polygons over $\Z 
[\zeta_{N}]$.  
 
  Next we describe the cusps on an arbitrary~$X_{G}$. For each cusp~$c$ 
of~$X_G$ we obtain a parameter at~$c$ on~$X_{G}$ by picking a lift $\tilde{c}$ of $c$ 
on $X(N)$ and taking the norm $\prod q_{\tilde{c}}^{1/N}\circ \gamma$, where~$\gamma$ runs 
through a set of representatives of $\Gamma /\Gamma (N)$ (recall that ${\Gamma :=\{ \gamma
\in\SL_2 (\Z ),\ (\gamma \mod N)\in G\}}$). We denote by~$t_{c}$ this parameter in the 
sequel. Note that it is defined over a (possibly strict) subring of $\Z [\zeta_{N}]$. The modular 
interpretation of~$X_{G}$ associates to each cusp an orbit of our enhanced N\'eron 
polygon ${\left(C,(q^{1/N},
\zeta_N )\right)}$ under the action of the group generated by~$G$ 
and~$\pm U$ given by ${\left( \topbot ac \topbot bd\right)\colon \left(C,(q^{1/N},
\zeta_N )\right) \mapsto\left(C,(q^{a/N}\zeta_{N}^b ,q^{c/N}\zeta_{N}^d ) \right)}$. It follows 
from the above that the cusps of~$X_G$ have values in a subring of~$\Z [\zeta_{N}]$. Moreover, assume that~$X_{G}$ is 
defined over~$K$, of which~$v$ is a place of characteristic~$p$, with~$N=p^n N'$ 
and~$p\nmid N'$. Extending $v$ to a place of $\OO_v [\zeta_{N'}]$ if necessary, and
setting ${\OO'_v :=(\OO [\zeta_{N'}])_v}$, one sees that the closed subscheme of cusps 
over~$\OO'_v$ may be written as a sum of connected components of shape $\mathrm{Spec}
(R)$ where~$R$ is a subring of ${\OO'_v [\zeta_{p^n}]}$. Therefore if ${v(N)=0}$, the 
subscheme of cusps is \'etale over~$\OO_v$, but this may not be the case if $v(N)>0$. In the
latter case, however, the ramification is well controlled. Indeed, with the preceding notations, 
set ${\pi :=(1-\zeta_{p^n})}$. Any two different $p^n$-th roots of unity~$\zeta_{p^n}^a$ 
and~$\zeta_{p^n}^b$ satisfy ${(\zeta_{p^n}^a -\zeta_{p^n}^b )=\pi^{p^k} \alpha}$ 
with~$\alpha$ a~$v$-invertible element and ${0\le k\le n-1}$. It follows that N\'eron 
polygons enhanced with a level-$N$ structure are distinct over ${\Z [\zeta_N ]/(\pi^{p^{n-1}
+1})}$. The modular interpretation shows more precisely that if two different cusps~$c_1$ 
and~$c_2$ have same reduction at~$v$, then~$t_{c_1} (c_2 )$ has $v$-adic valuation less or 
equal to~$1/(p-1)$ (if we normalize~$v$ to have ${v(p)=1}$). This remark will be used later on.  

   To illustrate all this with a familiar example, letting~$G:=\left( \topbot10\topbot **
\right) \subset \GL_2 (\Z /N\Z )$, which gives rise to the modular curve $X_1 (N)$, one
finds that there are $\left|(\Z /N\Z)^\times\right|$ cusps, with modular interpretation corresponding 
to ${\left\{ (C,\zeta_N^{\epsilon a} ): \epsilon\in \{\pm1\} \right\}}$ where~$a$ runs 
through ${(\Z  /N\Z )^\times /\pm 1}$, and ${\{(C,q^{\epsilon a /N}\zeta_N^{\alpha}) : \epsilon\in \{\pm1\},\ 
\alpha\in (\Z  /N\Z )\}}$, where~$a$ runs through the same set. 
The curve $X_1 (N)$ is defined over~$\Q$ and has a modular model over~$\Z$. The cusps 
in the former subset above have values in ${\Z \left[\zeta_{N} +\zeta_{N}^{-1}\right]}$, and 
the cusps in the latter subset have values in~$\Z$. In other words, the closed subscheme of 
cusps over~$\Z$ is isomorphic to the disjoint union of ${\mathrm{Spec} \left(\Z \left[\zeta_{N} +\zeta_{N}^{-1}\right]\right)}$ and $\left|(\Z /N\Z)^\times\right| /2$ copies of $\mathrm{Spec} (\Z )$.

  It is clear from the definition that the above parameter $t_{c}$ defines a $v$-analytic 
function on a $v$-adic neighborhood of~$c\in X_{G} (K_{v})$ which satisfies the initial 
condition ${t_{c} (c)=0}$. Further, if~$e_c$ is the ramification index of the covering $X_G
\to X(1)$ at~$c$ (clearly,~$e_c$ divides~$N$) then, setting ${q_{c} :=t_{c}^{e_{c}}}$, the
familiar expansion ${j= {q_{c}}^{-1}+ 744+196884{q_{c}}+\ldots}$ 
holds in a  $v$-adic neighborhood of~$c$, the right-hand side converging $v$-adically. 

To be precise,~$t_c$ and~$q_c$ are defined and analytic on the set ${\Omega_c=
\Omega_{c,v}}$ defined as follows. If~$v$ is archimedean then ${\Omega_c=Y_G(
\bar{K}_{v})\cup\{c\}}$; in other words,~$\Omega_c$ is~$X_G(\bar{K}_{v})$ with all 
the cusps except~$c$ taken away. If~$v$ is non-archimedean, then~$\Omega_c$ consists of 
the points from $X_G(\bar{K}_{v})$ having reduction~$c$ at~$v$. Notice that 
${X_G(\bar{K}_{v})=\bigcup_{c\in\CC}\Omega_c}$
if ${v\in M_K^\infty}$, and 
${\left\{P\in X_G(\bar{K}_{v}):|j(P)|_v>1\right\}=\bigcup_{c\in\CC}\Omega_c}$
if ${v\in M_K^0}$. More generally, since ${|j(q_c)|=|q_c|_v^{-1}}$ for a non-archimedean~$v$,  for any ${R\ge 
1}$ and any ${v\in M_K^0}$ we have
\begin{equation}
\label{egamomr}
\left\{P\in X_G(\bar{K}_{v}):|j(P)|_v>R\vphantom{{R^{-1}}}\right\}=\bigcup_{c
\in\CC}\left\{P\in \Omega_c:|q_c(P)|_v<R^{-1}\right\},
\end{equation}
which  will be used later.

If~$v$ is non-archimedean and ${v(N)=0}$, the sets~$\Omega_c$ are pairwise disjoint, as
in this case the cusps define a finite \'etale scheme over ${\cal O}_{v}$. In general 
however the sets~$\Omega_c$ are not disjoint, so we need to refine them in order to be able
to define the notion of ``$v$-nearest cusp". Put 
$$
R_v=
\begin{cases}
2500 &\text{if $v\in M_K^\infty$,}\\ 
1& \text{if $v\in M_K^0$  and ${v(N)=0}$,}\\ 
p^{N/(p-1)} & \text{if $v\in M_K^0$  and ${v\mid p \mid N}$,}
\end{cases} \qquad
r_v= \begin{cases}
0.001&\text{if $v\in M_K^\infty$,}\\
R_v^{-1}& \text{if $v\in M_K^0$.}
\end{cases}
$$
Finally,   put
$$
X_G(\bar{K}_{v})^+ =\left\{P\in X_G(\bar{K}_{v}): |j(P)|_v >R_v\right\}, 
\qquad \Omega_c^+=\Omega_{c,v}^+ =\{ P\in \Omega_c:|q_c(P)|_v < r_v\}.
$$
Notice that ${\Omega_c^+=\Omega_c}$ if~$v$ is non-archimedean and ${v(N)=0}$.

\begin{proposition}
\label{pnecu}
In the above set-up, the sets~$\Omega_c^+$ are pairwise disjoint and we have
\begin{equation}
\label{egamom}
X_G (\bar{K}_{v})^+\subseteq\bigcup_{c\in \CC}\Omega_c^+
\end{equation}
with equality for the non-archimedean~$v$.
\end{proposition}

The proposition implies that for every  ${P\in X_G (\bar{K}_{v})^+}$ there exists a 
unique cusp~$c$ such that ${P\in \Omega_c^+}$. We call it \emph{the $v$-nearest cusp} (or 
simply the nearest cusp) to~$P$.

\begin{proof}
Assume first that~$v$ is archimedean, so that $\bar{K}_{v}=\C$.  As above, 
let~$\Gamma$ be the pull-back of ${\SL_2(\Z/N\Z)\cap G}$ to ${\Gamma(1)=\SL_2(\Z)}$, 
and~$D$ be the usual fundamental domain for $\Gamma(1)$. Notice that $\Gamma(1)$ acts 
naturally on the set~$\CC$ of cusps, and that ${q_{\gamma(c)}=q_c\circ\gamma^{-1}}$ for 
${\gamma\in \Gamma(1)}$, which implies that ${\gamma(\Omega_c^+)=
\Omega_{\gamma(c)}^+}$. 

Fix ${P\in X_G(\C)=\bar\HH/\Gamma}$, and pick a representative ${\tau\in\bar\HH}$ 
for~$P$. Then there exists ${\gamma\in \Gamma(1)}$ such that ${\gamma(\tau)\in 
\tilD}$, where we put
${\tilD=D\cup\{i\infty\}}$. 
With the common abuse of notation, we denote by~$j$ both the  $j$-invariant on~$\HH$ and 
on~$X_G$, so that ${j\bigl(\gamma(\tau)\bigr)=j(\tau)=j(P)}$. Now if ${P\in X_G
(\C)^+}$ then 
${\left|j\bigl(\gamma(\tau)\bigr)\right|=|j(P)|>2500}$, 
and Corollary~\ref{cdplus} implies that ${P^\gamma \in \Omega_{c_\infty}^+}$. Hence 
${P\in \Omega_c^+}$, where ${c=\gamma^{-1}(c_\infty)}$. This 
proves~(\ref{egamom}).

Now let us prove that the sets~$\Omega_c^+$ are pairwise disjoint. For ${\tau\in \HH}$ the 
condition ${|q_\tau|<0.001}$ implies that ${\im\tau>1}$. Hence
$$
\{\tau\in \bar\HH: |q_\tau|<0.001\}\subset \bigcup_{\genfrac{}{}{0pt}{}{\gamma
\in \Gamma(1)}{\gamma(i\infty)=i\infty}}\gamma\left(\tilD\right).
$$
It follows that the pull-back of $\Omega_c^+$ to~$\bar\HH$ is contained in the set
${\Delta_c=\bigcup_{\genfrac{}{}{0pt}{}{\gamma\in \Gamma(1)}{\gamma(c_\infty)
=c}}\gamma\left(\tilD\right)}$.
By the definition of~$\tilD$, for ${\gamma\ne \pm1}$ we have
${\gamma(\tilD)\cap \tilD=\{i\infty\}}$  if $\gamma(i\infty)=i\infty$, and ${\gamma(\tilD)\cap \tilD=
\varnothing}$ otherwise.
It follows that the sets~$\Delta_c$ are pairwise disjoint. Hence so are the 
sets~$\Omega_c^+$. This completes the proof for archimedean~$v$.

We now assume that~$v$ is non-archimedean. In this case~(\ref{egamom}) holds, with 
equality, as a particular case of~(\ref{egamomr}), and we only need to show that the 
sets~$\Omega_c^+$ are pairwise disjoint. 
If $v(N)=0$ then, as already mentioned, the cusps of $X_{G}$ define a finite \'etale closed 
subscheme of $X_{G}$ over ${\cal O}_{v}$, so the sets $\Omega_c=\Omega_c^+$ are 
obviously pairwise disjoint.

Now assume ${v(N)>0}$. Let~$p$ be the residue characteristic of~$v$ and~$p^n \| N$ be 
the largest power of~$p$ dividing~$N=p^n N'$. As we have seen, the scheme of cusps 
on~$X_{G}$ may be no longer \'etale over~$\OO_v$. We can however still partition it into 
connected components, which totally ramify in the fiber at~$v$. More precisely, setting as 
above~${\OO'}_{v} :=(\OO [\zeta_{N'}])_v$, each connected component over $\OO'_{v}$
is schematically a~$\mathrm{Spec} (R)$ for~$R$ a subring of~${\OO'}_v [\zeta_{p^n}]$. 
Each set~$\Omega_{c}$ contains exactly one such connected component of cusps, so 
when~$R$ does ramify nontrivially at~$v$, then~$\Omega_{c}$ is clearly ``too large'' (one 
has~$\Omega_{c_{1}} =\Omega_{c_{2}}$ exactly when~$c_{1}$ and~$c_{2}$ have same 
reduction at~$v$). We want to show that, nevertheless, the refined sets~$\Omega_c^+$ are 
pairwise disjoint.

  If the cusps~$c_1$ and~$c_2$ belong to distinct connected components, then 
already~$\Omega_{c_{1}}$ and~$\Omega_{c_{2}}$ are disjoint, so~$\Omega_{c_{1}}^+$ 
and~$\Omega_{c_{2}}^+$ are disjoint \emph{a fortiori}. Now assume that ~$c_1$ 
and~$c_2$ belong to the same component, i.e. have same reduction at~$v$. In this case, as 
explained before the proposition, one may write ${t_{c_{1}}(c_{2} )=\pi^{p^k} a\in 
\OO'_v [\zeta_{p^n} ]}$, for $\pi$ a certain uniformizer (e.g. $\pi :=(\zeta_{p^n}-1)$), where
the element~$a$ is~$v$-invertible and ${0\le k\le n-1}$. As $v(\pi )=1/p^{n-1} (p-1)$ and 
$\Omega_{c_1}^+$ is contained in $\left\{ P\in X_{G} (\bar{K}_v ):|t_{c_1}(P)|_v <
p^{-1/(p-1)}\right\}$, we see that~$c_2$ does not belong to~$\Omega_{c_1}^+$, which implies 
that the sets~$\Omega_{c_{1}}^+$ and~$\Omega_{c_{2}}^+$ are disjoint. This completes 
the proof of the  proposition. \qed
\end{proof}

\medskip

Now Propositions~\ref{pqj}  has the following consequence. 

\begin{proposition}
If~$v$ is archimedean then for ${P\in X_G(K_v)^+}$ with the nearest cusp~$c$  we have 
${\left|j(P)-q_c(P)^{-1}-744\right|_v \le 330000|q_c(P)|_v}$.
In particular, ${\left|j(P)-q_c(P)^{-1}\right|_v\le 1100}$ and
\begin{equation}
\label{everysimple}
\frac32 |j(P)|_v\ge \left|q_c(P)^{-1}\right|_v\ge \frac12 |j(P)|_v. 
\end{equation}
\end{proposition}

\section{Modular Units}
\label{smuni}

In this section we  recall the construction of modular units on the modular curve $X_G$. By 
a \emph{modular unit} we mean a rational function on~$X_G$ having poles and zeros only 
at the cusps. 

\subsection{Integrality of Siegel's Function}

For ${\bfa\in \Q^2\setminus\Z^2}$ Siegel's function~$g_\bfa$ (see Subsection~\ref{sssieg}) is algebraic over the field $\C(j)$: this follows from the fact that $g_\bfa^{12}$ is automorphic of level $2N^2$ \cite[page~29]{KL81}. Since~$g_\bfa$ is holomorphic and does not vanish on the upper half-plane~$\HH$, both~$g_\bfa$ and $g_\bfa^{-1}$  must be integral over the ring $\C[j]$. Actually, a stronger assertion holds.

\begin{proposition}
\label{psiu}
Both~$g_\bfa$ and   ${\left(1-\zeta_N\right)g_\bfa^{-1}}$ are integral over  $\Z[j]$. Here~$N$ is the exact order of~$\bfa$ in $(\Q/\Z)^2$ and~$\zeta_N$ is a primitive $N$-th root of unity. 
\end{proposition}

This  is, essentially, established in~\cite{KL81}, but is not stated explicitly therein. Therefore we briefly indicate the proof here. 
Recall that a holomorphic and $\Gamma(N)$-automorphic function ${f:\HH\to \C}$  admits the \emph{infinite $q$-expansion}
\begin{equation}
\label{eexpan}
f(\tau)= \sum_{k\in \Z}a_kq^{k/N},
\end{equation}
where ${q=q_\tau=e^{2\pi i\tau}}$. 
We call the $q$-series~(\ref{eexpan}) \emph{algebraic integral} if the following two conditions are satisfied: the negative part of~(\ref{eexpan}) has only finitely many terms 
(that is, ${a_k=0}$ for large negative~$k$), and the  coefficients~$a_k$ are algebraic 
integers. Algebraic integral $q$-series form a ring. The invertible elements of this ring are 
$q$-series with invertible leading coefficient. By the \emph{leading coefficient} of an 
algebraic integral $q$-series we mean~$a_m$, where  ${m\in \Z}$ is defined by ${a_m\ne 
0}$, but ${a_k=0}$ for ${k<m}$.

\begin{lemma}
\label{lint}
Let~$f$ be a $\Gamma(N)$-automorphic function  such that for every ${\gamma \in \Gamma
(1)}$ the $q$-expansion of ${f\circ \gamma}$ is algebraic integral. Then~$f$ is integral 
over $\Z[j]$. 
\end{lemma}

\begin{proof}
This is, essentially, Lemma~2.1 from \cite[Section~2.2]{KL81}. Since~$f$ is 
$\Gamma(N)$-automorphic, the set ${\{f\circ\gamma :\gamma\in \Gamma(1)\}}$ is finite. 
The coefficients of the polynomial ${F(T)=\prod(T-f\circ\gamma)}$ (where the product is 
taken over the finite set above)  are $\Gamma(1)$-automorphic functions  with algebraic 
integral $q$-expansions. By the $q$-expansion principle, the coefficients of $F(T)$ belong to 
$\bar\Z[j]$, where~$\bar\Z$ is the ring of all algebraic integers. It follows that~$f$ is 
integral over~$\bar\Z[j]$, hence over $\Z[j]$. \qed
\end{proof}

\paragraph*{Proof of Proposition~\ref{psiu}}
The function $g_\bfa^{12}$ is automorphic of level $2N^2$  and its 
$q$-expansion is algebraic integral (as one can easily see by transforming the infinite 
product~(\ref{epga}) into an infinite series).  By~(\ref{eperga}), the same is true for for 
every ${(g_\bfa\circ\gamma)^{12}}$. Lemma~\ref{lint} now implies that $g_\bfa^{12}$ is integral 
over $\Z[j]$, and so is~$g_\bfa$. 

Further, the  $q$-expansion of~$g_\bfa$ is invertible if ${a_1\notin \Z}$ and is 
${1-e(a_2)}$ times an invertible $q$-series if ${a_1\in \Z}$. Hence the $q$-expansion 
of~$g_\bfa^{-1}$ is algebraic integral when ${a_1\notin \Z}$, and if ${a_1\in \Z}$ the 
same is true for ${\left(1-e(a_2)\right)g_\bfa^{-1}}$. In the latter case~$N$ is the exact 
order of~$a_2$ in $\Q/\Z$, which implies that ${(1-\zeta_N)/\left(1-e(a_2)\right)}$ is 
an algebraic unit. Hence, in any case, ${(1-\zeta_N) g_\bfa^{-1}}$ has algebraic integral 
$q$-expansion, and the same is true with~$g_\bfa$ replaced by ${g_\bfa\circ\gamma}$ for
any ${\gamma\in \Gamma(1)}$ (we again use~(\ref{eperga}) and notice that~$\bfa$ and 
$\bfa\gamma$ have the same order in $(\Q/\Z)^2$). Applying Lemma~\ref{lint} to the function  ${\left((1-\zeta_N) g_\bfa^{-1}\right)^{12}}$,   we 
complete the proof.\qed

\subsection{Modular Units on $X(N)$}
\label{ssxn}

From now on, we fix an integer ${N>1}$. Recall that the curve $X(N)$ is defined over the 
field $\Q(\zeta_N)$. Moreover, the field ${\Q\bigl(X(N)\bigr)=\Q(\zeta_N)\bigl(X(N)\bigr)}$  is a Galois 
extension of $\Q(j)$, the Galois group being isomorphic to $\GL_2(\Z/N\Z)$. The 
isomorphism 
\begin{equation}
\label{eisog}
\gal\left(\left.\Q\bigl(X(N)\bigr)\right/\Q(j)\right)\cong \GL_2(\Z/N\Z)
\end{equation}
is defined up to an inner automorphism; once it is fixed, we have the well-defined 
isomorphisms 
\begin{equation}
\gal\left(\left.\Q\bigl(X(N)\bigr)\right/\Q(\zeta_N,j)\right)\cong
\SL_2(\Z/N\Z),\qquad
\label{eisoc}
\gal\bigl(\left.\Q(\zeta_N)\right/\Q\bigr)\cong (\Z/N\Z)^\times,
\end{equation}
and we may  identify the groups on the left and on the right in~(\ref{eisog} and~\ref{eisoc}). 
Our choice of the isomorphism~(\ref{eisog}) will be specified in Proposition~\ref{pua}.

According to Theorem~1.2 from \cite[Section~2.1]{KL81}, given ${\bfa=(a_1,a_2) \in
(N^{-1}\Z)^2\setminus \Z^2}$, the function $g_\bfa^{12N}$ is $\Gamma
(N)$-automorphic of weight~$0$. 
Hence $g_\bfa^{12N}$ defines a rational function on the modular curve $X(N)$, to be denoted  by~$u_\bfa$. Since the root of 
unity in~(\ref{eaa'}) is of order dividing $12N$, we have ${u_\bfa=u_{\bfa'}}$ when 
${\bfa\equiv \bfa'\mod \Z^2}$. Hence~$u_\bfa$ is well-defined when~$\bfa$ is a 
non-zero element of the abelian group ${(N^{-1}\Z/\Z)^2}$, which will be assumed in the 
sequel. We put
${\bfA=(N^{-1}\Z/\Z)^2\setminus\{0\}}$.

The functions~$u_\bfa$ have the following properties. 

\begin{proposition}
\label{pua}
\begin{enumerate}
\item
\label{iintmun}
The functions~$u_\bfa$ and ${(1-\zeta_{N_\bfa})^{12N}u_\bfa^{-1}}$ are integral 
over $\Z[j]$, where~$N_\bfa$ is the exact order of~$\bfa$ in ${(N^{-1}\Z/\Z)^2}$. 
In particular,~$u_\bfa$ has zeros and poles only at the cusps of $X(N)$.

\item
\label{ifrick}
The functions~$u_\bfa$ belong to the field ${\Q\bigl(X(N)\bigr)}$,
and the Galois action on the set $\{u_\bfa\}$ over $\Q(j)$ is compatible with the (right) 
linear action of $\GL_2(\Z/N\Z)$ on~$\bfA$  in the following sense: the 
isomorphism~(\ref{eisog}) can be chosen so that for any 
${\sigma\in \gal\Bigl(\left.\Q\bigl(X(N)\bigr)\right/\Q(j)\Bigr)=\GL_2
(\Z/N\Z)}$
and any ${\bfa\in \bfA}$  we have 
${u_\bfa^\sigma=u_{\bfa\sigma}}$.

\item
\label{iord}
For the cusp~$c_\infty$ at infinity we have ${\ord_{c_\infty}u_\bfa=12N^2
\ell_\bfa}$, where~$\ell_\bfa$ is defined in Subsection~\ref{ssnota}. For an arbitrary cusp~$c$ we have
${\left|\ord_cu_\bfa\right|\le N^2}$.

\end{enumerate}

\end{proposition}

\begin{proof}
Item~(\ref{iintmun}) follows from Proposition~\ref{psiu}. 
Item~(\ref{ifrick}) is Proposition~1.3 from 
\cite[Chapter~2]{KL81}. 
We are left with item~(\ref{iord}). The order of vanishing of~$u_\bfa$ at $i\infty$ is 
$12N\ell_\bfa$. Since the ramification index of the covering ${X(N)\to X(1)}$ at every 
cusp is~$N$, we obtain ${\ord_{c_\infty}u_\bfa=12N^2\ell_\bfa}$. Since
${|\ell_\bfa|\le 1/12}$, we have ${\left|\ord_{c_\infty}u_\bfa
\right|\le N^2}$. The case of arbitrary~$c$ reduces to the case ${c=c_\infty}$ upon 
replacing~$\bfa$ by $\bfa\sigma$ where  ${\sigma\in \GL_2(\Z/N\Z)}$ is such that 
${\sigma(c)=c_\infty}$.    \qed
\end{proof}

\medskip

The group generated by the principal divisors ${(u_\bfa)}$, where ${\bfa\in \bfA}$, is 
contained in  the group of cuspidal divisors on~$X(N)$ (that is, the divisors  supported at the
set ${\CC(N)=\CC(\Gamma(N))}$ of cusps). Since principal divisors are of degree~$0$, the
rank of the former group is at most ${|\CC(N)|-1}$. It is fundamental for us that this rank is 
indeed maximal possible.  The following proposition is Theorem~3.1 in \cite[Chapter~2]{KL81}.

\begin{proposition}
\label{pmdrxn}
The  group generated by the set
${\left\{(u_\bfa):\bfa\in \bfA\right\}}$ is of rank
${|\CC(N)|-1}$. \qed
\end{proposition}

We also need to know the  behavior of the functions~$u_\bfa$ near the cusps, and estimate them in terms of the modular invariant~$j$. In the following proposition~$K$ is a number field containing~$\zeta_N$ and~$v$ is a valuation of~$K$, extended somehow to~$\bar K$. We use the notation of Section~\ref{snecu}.

\begin{proposition}
\label{ploc}
\begin{enumerate}
\item
Let~$c$ be a cusp of $X(N)$.
If ${v\in M_K^\infty}$ then  
\begin{align*}
\bigl|\log |u_\bfa(P)|_v- \ord_cu_\bfa\log|t_c(P)|_v\bigr| &\le 36N|q_c(P)|_v^{1/N} &&
\text{when  ${a_1\ne0}$,}\\
\bigl|\log |u_\bfa(P)|_v- \ord_cu_\bfa\log|t_c(P)|_v-12N\log|1-e(a_2)|_v\bigr| &\le 36N|q_c(P)|_v &&
\text{when ${a_1=0}$}
\end{align*}
for any ${P\in \Omega_{c,v}}$ such that ${|q_c(P)|_v<10^{-N}}$.
If ${v\in M_K^0}$ then 
$$
\log |u_\bfa(P)|_v=
\begin{cases} \ord_cu_\bfa\log|t_c(P)|_v  &
\text{when  ${a_1\ne0}$,}\\ 
 \ord_cu_\bfa\log|t_c(P)|_v+12N\log|1-e(a_2)|_v &
\text{when ${a_1=0}$}
\end{cases}
$$
for any  ${P\in \Omega_{c,v}}$.

\item
If ${v\in M_K^\infty}$ then  
$$
\bigl|\log |u_\bfa(P)|_v\bigr|\le N\log\bigl(|j(P)|_v+2200\bigr)+14N\log N
$$ 
for any ${P\in X(N)(K_v)}$. If ${v\in M_K^0}$ then  
${\bigl|\log |u_\bfa(P)|_v\bigr|\le N\log|j(P)|_v+\log R_v}$ for any ${P\in X(N)(K_v)}$ such that ${|j(P)|_v>1}$. 

\end{enumerate}
\end{proposition}

\begin{proof}
When ${c=c_\infty}$ this is an immediate consequence of Propositions~\ref{psiar},~\ref{pgalar},~\ref{psinar} (notice that ${\log| q_c|v_=N\log |t_c|_v}$ for every cusp~$c$) and Corollary~\ref{cnaglob}. The general case reduces to the case ${c=c_\infty}$ by applying a suitable Galois automorphism. \qed
\end{proof}

\subsection{$K$-rational Modular Units on $X_G$}
\label{ssmunk}

Now let~$K$ be a number field, and let~$G$ be a subgroup of $\GL_2(\Z/N\Z)$. Let 
$\det G$ be the image of~$G$ under the determinant map ${\det \colon \GL_2(\Z/N\Z)
\to (\Z/N\Z)^\times=\gal(\Q(\zeta_N)/\Q)}$ (recall that we have a well-defined isomorphism~(\ref{eisoc})).
In the sequel we shall assume that 
${K\supseteq\Q(\zeta_N)^{\det G}}$,
where $\Q(\zeta_N)^{\det G}$ is the subfield of $\Q(\zeta_N)$ stable under $\det G$. 
This assumption implies that  the curve~$X_G$ is defined over~$K$. Then ${G':=\gal\left(\left.K\bigl(X(N)\bigr)\right/K\left(X_G \right)\right)}$
is a subgroup of~$G$. For every ${\bfa\in \bfA}$ we put
${w_\bfa= \prod_{\sigma\in G'}u_{\bfa \sigma}}$. 
Since ${u_{\bfa \sigma}=u_\bfa^\sigma}$, 
the functions~$w_\bfa$ are contained in $K(X_G)$.  They have the following properties. 

\begin{proposition}
\label{pwa}
\begin{enumerate}
\item
\label{imunk}
The functions~$w_\bfa$ have zeros and poles only at the cusps of $X_G$. If~$c$ is such a 
cusp, then ${\left|\ord_cw_\bfa\right|\le |G'|N^2}$. 
\item
\label{iintk}
For every ${\bfa \in \bfA}$ there exists an algebraic integer ${\lambda_\bfa \in \Z
[\zeta_N]}$, which is a product of $|G'|$ factors of the form ${\left(1-\zeta_{N'}
\right)^{12N}}$, where ${N'\mid N}$, such that the functions~$w_\bfa$ and 
${\lambda_\bfa w_\bfa^{-1}}$ are integral over~$\Z[j]$. 

\item
\label{icrk}
If ${v\in M_K^\infty}$  then  
$$
\bigl|\log |w_\bfa(P)|_v\bigr|\le |G'|N \log\bigl(|j(P)|_v+2200\bigr)+14|G'|N \log 
N
$$
for any ${P\in X_G(K_v)}$. 
If ${v\in M_K^0}$  then 
$$
\bigl|\log |w_\bfa(P)|_v\bigr|\le |G'|N\log|j(P)|_v+|G'|\log R_v
$$ 
for any ${P\in X(N)(K_v)}$ such that ${|j(P)|_v>1}$.

\item
\label{iloc}
For every ${\bfa \in \bfA}$ and every cusp~$c$ there exists an algebraic integer ${\beta=
\beta(\bfa,c)\in \Z[\zeta_N]}$, which is a product of at most $|G'|$ factors of the form 
${(1-e(a))^{12N}}$, where ${a\in N^{-1}\Z/\Z}$ and ${a\ne0}$, such that for any ${v\in 
M_K}$ and for any ${P\in \Omega_{c,v}}$ we have the following. If~$v$ is archimedean 
and ${|q_c(P)|_v\le 10^{-N}}$ then
$$
\bigl|\log\left|w_\bfa(P)\right|_v- \ord_cw_\bfa\log|t_c(P)|_v- \log|\beta|_v
\bigr| \le 36|G'|N |q_c(P)|_v^{1/N}.
$$
If~$v$ is non-archimedean then 
${\log\left|w_\bfa(P)\right|_v= \ord_c w_\bfa\log|t_c(P)|_v+ \log|\beta|_v}$.

\item
\label{imdk}
The group generated by the principal divisors $(w_\bfa)$ is of rank ${|\CC(G,K)|-1}$. 

\end{enumerate}
\end{proposition}

\begin{proof}
Items~(\ref{imunk}) and~(\ref{iintk})  follow from  Proposition~\ref{pua}, 
items~(\ref{icrk}) and~(\ref{iloc}) follow from Proposition~\ref{ploc}. Finally, 
item~(\ref{imdk}) follows from Proposition~\ref{pmdrxn} and Lemma~\ref{ltriv} 
below. On should apply the lemma (whose proof is left to the reader) with~$A$ as the group of degree~$0$ cuspidal divisors 
on $X(N)$, with~$B$ as the group of all cuspidal divisors on $X(N)$ generated by the 
principal divisors $(u_\bfa)$, and with~$G$ as~$G'$. \qed
\end{proof}

\begin{lemma}
\label{ltriv}
Let~$G$ be a finite group, and let~$A$ be a torsion-free finitely generated (left) $G$-module. 
For ${a\in A}$ put 
${a_G=\sum_{g\in G}ga}$ 
and denote by $A^G$  the submodule of the $G$-invariant elements. Further, let~$B$ be a 
finite index subgroup of~$A$. Then ${B_G=\{b_G:b\in B\}}$ is a finite index submodule 
of~$A^G$. \qed
\end{lemma}

\subsection{A Unit Vanishing at the Given Cusps}

Item~(\ref{imdk}) of Proposition~\ref{pwa} implies that for any proper subset of 
$\CC(G,K)$ there is a $K$-rational unit on~$X_G$  vanishing at this subset. In this 
subsection we give a quantitative version of this fact. 
We shall use the following simple lemma, where we denote by $\|\cdot\|_1$ the $\ell_1$-norm.

\begin{lemma}
\label{llia}
Let~$M$ be an  ${s\times t}$ matrix of rank~$s$ with entries in~$\Z$ . Assume that the 
entries of~$M$ do not exceed~$A$ in absolute value. Then there exists a vector  ${\bfb\in 
\Z^t}$ such that ${\|\bfb
\|_1\le s^{s/2+1}A^{s-1}}$,  and such that all the~$s$ coordinates of the vector $M\bfb$ 
(in the standard basis) are strictly positive. 
\end{lemma}

\begin{proof}
Assume first that ${s=t}$. Let~$d$ be the determinant of~$M$. Then the column vector 
${(|d|, \ldots,|d|)}$ can be written as $M\bfb$, where ${\bfb=(b_1, \ldots, b_s)}$ 
with~$b_k$ being (up to the sign) the determinant of the matrix obtained from~$M$ upon 
replacing the $k$-th row by ${(1, \dots,1)}$. Using Hadamard's inequality, we bound 
$|b_k|$ by $\sqrt s\left(\sqrt s A\right)^{s-1}$.  This proves the lemma in the case
${s=t}$. The general case reduces to the case ${s=t}$ by selecting a non-singular ${s\times 
s}$ sub-matrix, which gives~$s$ entries of the vector~$\bfb$; the remaing ${t-s}$ entries are
set to be~$0$.  \qed
\end{proof}

\medskip

Now let~$G$,~$K$ and~$G'$ be as in Subsection~\ref{ssmunk}.

\begin{proposition}
\label{pw}

\begin{sloppypar}
Let~$\Sigma$ be a proper subset of $\CC(G,K)$. Assume that ${|\Sigma|\le s}$, and put ${B=s^{s/2+1}\left(|G'|N^2\right)^{s-1}}$. 
Then there exists a $K$-rational modular unit~$w$ on~$X_G$ 
with the following properties.
\end{sloppypar}

\begin{enumerate}
\item If~$c$ is a cusp such that the orbit of~$c$ is in~$\Sigma$ then ${\ord_cw>0}$.

\item
\label{icuw}
For every cusp~$c$ we have 
\begin{equation}
\label{eordw}
\left|\ord_cw\right|\le B|G'|N^2.
\end{equation}

\item
\label{ilamb} There exists an algebraic integer~$\lambda$, which is a product of at most 
$|G'|B$ factors of the form ${\left(1-\zeta_{N'}\right)^{12N}}$, where ${N'\mid N}$, 
such that $\lambda w$ is integral over $\Z[j]$.  

\item
\label{icrw}
If ${v\in M_K^\infty}$  then for any ${P\in X_G(K_v)}$ we have  
$$
\bigl|\log |w(P)|_v\bigr|\le B|G'|N \log\bigl(|j(P)|_v+2200\bigr)+14B|G'|N \log N.
$$
If ${v\in M_K^0}$ then for any ${P\in X(N)(K_v)}$ such that ${|j(P)|_v>1}$ we 
have  
$$
\bigl|\log |w(P)|_v\bigr|\le B|G'|N\log|j(P)|_v+B|G'|\log R_v. 
$$

\item
\label{iword}
For every cusp~$c$ there exists an algebraic number ${\beta=\beta(c)\in \Z[\zeta_N]}$
which is a product of at most $|G'|B$ factors of the form ${(1-e(a))^{\pm12N}}$, where ${a
\in N^{-1}\Z/\Z}$ and ${a\ne0}$, such that for any ${v\in M_K}$ and for any ${P\in
\Omega_{c,v}}$ we have the following. If~$v$ is archimedean and ${|q_c(P)|_v\le 
10^{-N}}$ then
$$
\bigl|\log\left|w(P)\right|_v- \ord_cw\log|t_c(P)|_v- \log|\beta|_v\bigr| \le 
36B|G'|N |q_c(P)|_v^{1/N}.
$$
If~$v$ is non-archimedean then 
${\log\left|w(P)\right|_v= \ord_c w\log|t_c(P)|_v+ \log|\beta|_v}$. 

\end{enumerate}

\end{proposition}

\begin{proof}
The $K$-rational Galois orbit of a cusp~$c$ has ${[K(c):K]}$ elements. 
Fix a representative in every such orbit  and 
consider the ${|\CC(G,K)|\times|\bfA| }$ matrix $\left(\ord_c w_\bfa\right)$, 
where~$c$ runs over the set of selected representatives. According to item~(\ref{imdk}) of 
Proposition~\ref{pwa}, this matrix is of rank ${|\CC(G,K)|-1}$, and the only (up to 
proportionality) linear relation between the rows  is 
${\sum_c[K(c):K]\ord_cw_\bfa=0}$ for  every ${\bfa\in \bfA}$.
It follows that any proper subset of the rows of our matrix is linearly independent. In particular, if we select  the rows corresponding to the set~$\Sigma$, we get a sub-matrix  of rank~$|\Sigma|$. Applying to it Lemma~\ref{llia}, where we may take ${A=|G'|N^2}$ due to item~(\ref{imunk}) of Proposition~\ref{pwa}, we find integers~$b_\bfa$  such that ${\sum_{\bfa\in \bfA}|b_\bfa|\le  B}$ and such that the  function ${w=\prod_{\bfa\in \bfA}w_\bfa^{b_\bfa}}$ is as wanted. \qed
\end{proof}

\section{Proof of Theorem~\ref{tbo}}
\label{sproof}

We use the notation of Section~\ref{snecu}. We put 
${R_v'=50^N}$ for archimedean~$v$ and ${R_v'=R_v}$  for non-archimedean~$v$. Since ${N\ge 2}$, we have ${R_v'\ge R_v}$. 

We use the notation ${d_v=[K_v:\Q_v]}$ and ${d=[K:\Q]}$. We fix an extension of every ${v\in M_K}$ to~$\bar K$ and denote this extension by~$v$ as well.

We shall use the estimate 
\begin{equation}
\label{ecalrest}
\calR \le \sum_{p\mid N}\frac{\log p}{p-1}\le  \omega(N)\log 2\le \log N, 
\end{equation}
for the quantity~$\calR$, defined in~(\ref{ecalr}).
(Here~$\omega(N)$ is the number of prime divisors of~$N$.) Of course, much sharper 
estimates for~$\calR$  are possible as well, but~(\ref{ecalrest}) is plainly sufficient for us.

\subsection{The Runge Unit}
Fix ${P\in Y_G(\OO_S)}$. Let~$S_1$ consist of the places ${v\in M_K}$ such that ${|j(P)|_v>R_v'}$. Plainly, ${S_1\subset S}$. Since ${R_v'\ge R_v}$, Proposition~\ref{pnecu} applies to our~$P$ and every ${v\in S_1}$. Thus, 
for ${v\in S_1}$ let~$c_v$ be the $v$-nearest cusp to~$P$, and let~$\Sigma$ be the set of all $\gal(\bar K/K)$-orbits of cusps containing some of the~$c_v$. Then
${|\Sigma|\le |S_1|\le |S|}$,
and since ${|S|< |\CC(G,K)|}$ by the assumption,~$\Sigma$ is a proper subset of $\CC(G,K)$. Let~$w$ and~$B$ be as in Proposition~\ref{pw}, where we may put ${s=|S|}$. Then ${\ord_{c_v}w>0}$ for every ${v\in S_1}$, and the other statements of this proposition are satisfied.

Since~$w$ is a modular unit and~$P$ is not a cusp, we have ${w(P)\ne 0, \infty}$, and the product formula gives
${\sum_{v\in M_K}d_v\log |w(P)|_v=0}$.
We want to show that this is impossible when $\height(P)$ is too large.

\subsection{Partitioning the Places of~$K$}
We partition the set of places~$M_K$ into three pairwise disjoint subsets:
${M_K=S_1\cup S_2\cup S_3}$, where ${S_i\cap S_j =\varnothing}$  for ${i\ne j}$.
The set~$S_1$ is already defined. The set~$S_2$ consists of the \emph{archimedean} places not belonging to~$S_1$ and the \emph{non-archimedean} places~$v$ not belonging to~$S_1$ and such that ${|j(P)|_v>1}$. (Obviously, ${S_2\subset S}$.) Finally, the set~$S_3$ consists of the places not belonging to ${S_1\cup S_2}$; in other words, ${v\in S_3}$ if and only if~$v$ is non-archimedean and ${|j(P)|_v\le 1}$. 

We will estimate from above the three sums
${\Xi_i=\sum_{v\in S_i}d_v\log |w(P)|_v}$. 
We will show that ${\Xi_1 \le -N^{-1}d\height(P)+O(1)}$, where the $O(1)$-term is independent of~$P$ (it will be made explicit). Further, we will bound $\Xi_2$ and $\Xi_3$  independently of~$P$. Since 
\begin{equation}
\label{esss}
\Xi_1+\Xi_2+\Xi_3=0,
\end{equation}
(which is a different writing of the product formula), this would bound $\height(P)$.

\subsection{Estimating $\Xi_1$}
For ${v\in S_1}$ we have ${P\in \Omega_{c_v,v}}$,  we may apply item~(\ref{iword}) of Proposition~\ref{pw}. Since ${\ord_{c_v}w>0}$ and ${\log q_{c_v}(P)= e\log t_{c_v}(P)}$ with ${e\mid N}$, we have, for an  archimedean ${v\in S_1}$
\begin{align}
\log\left|w(P)\right|_v &\le  \frac{\ord_{c_v}w}e\log|q_{c_v}(P)|_v+ \log|\beta(c_v)|_v + 36B|G'|N |q_{c_v}(P)|_v^{1/N} \nonumber\\
\label{ewple}
&\le  -\frac{\ord_{c_v}w}N\log|j(P)|_v+ \log|\beta(c_v)|_v + 2B|G'|N \qquad\text{(we use~(\ref{everysimple}) and~(\ref{eordw}))}\\
\label{exioar}
& \le  -N^{-1}\log|j(P)|_v+ \log|\beta(c_v)|_v + 2B|G'|N .
\end{align}
For a non-archimedean ${v\in S_1}$ we have
\begin{equation}
\label{exionar}
\log\left|w(P)\right|_v\le N^{-1} \log|q_c(P)|_v+ \log|\beta(c)|_v= -N^{-1}\log|j(P)|_v+ \log|\beta(c_v)|_v.
\end{equation}

Next, we want to estimate $\sum_{v\in S_1}\log|\beta(c_v)|_v$. Recall that $\beta(c)$ is a product of at most $12B|G'|N $ numbers of the type ${1-e(a)}$, where~$a$ is a non-zero element of ${N^{-1}\Z/\Z}$. For such~$a$ we have
${1/N \le \bigl|1-e(a)\bigr|_v\le 2}$ if~$v$ is archimedean,
${p^{-1/(p-1)}\le \bigl|1-e(a)\bigr|_v\le 1}$  if~$v$ is non-archimedean and ${v\mid p\mid N}$, and ${\bigl|1-e(a)\bigr|_v=1}$ if~$v$ is non-archimedean and ${v(N)=0}$. 
It follows that 
\begin{equation}
\label{esumbet}
\sum_{v\in S_1}\log|\beta(c_v)|_v \le 12dB|G'|N \left(\log N+\calR\right)\le 24dB|G'|N \log N,
\end{equation}
where~$\calR$ is defined in~(\ref{ecalr}) and is estimated using~(\ref{ecalrest}). 

Thus, combining~(\ref{exioar}),~(\ref{exionar}) and~(\ref{esumbet}), we obtain 
\begin{align*}
\Xi_1&\le -N^{-1}\sum_{v\in S_1}d_v\log|j(P)|_v+\sum_{v\in S_1}\log|\beta(c_v)|_v+2dB|G'|N \\
&\le -N^{-1}\sum_{v\in S_1}d_v\log|j(P)|_v+27dB|G'|N \log N. 
\end{align*}
Further, since ${|j(P)|_v\le R_v'}$ for ${v\in S\setminus S_1}$, we have 
$$
\sum_{v\in S\setminus S_1}d_v\log|j(P)|_v \le \sum_{v\in M_K}d_v\log R_v'\le dN \left(\log 50+\calR\right)\le dN\log(50N),
$$
by~(\ref{ecalrest}), and we obtain 
\begin{align*}
\Xi_1
&\le -N^{-1}\sum_{v\in S}d_v\log|j(P)|_v+dN \bigl(27|G'|B\log N +\log(50N)\bigr)\\
&\le -N^{-1}\sum_{v\in S}d_v\log|j(P)|_v+ 32dB|G'|N \log N.
\end{align*}
Finally, since $j(P)$ is an $S$-integer, we have 
$$
\height(P)=\height(j(P))=d^{-1}\sum_{v\in S}d_v\log^+|j(P)|_v\ge d^{-1}\sum_{v\in S}d_v\log|j(P)|_v,
$$
and we obtain
\begin{equation}
\label{esone}
\Xi_1\le d\left(-N^{-1}\height(P)+ 31B|G'|N \log N\right).
\end{equation}

\subsection{Estimating $\Xi_2$, $\Xi_3$ and Completing the Proof}
Item~(\ref{icrw}) of Proposition~\ref{pw} implies that for an archimedean ${v\in S_2}$ 
$$
\log |w(P)|_v \le B|G'|N\left(\log \left(50^N+2200\right)+14\log N\right)\le 10B|G'|N^2,
$$
and for a non-archimedean ${v\in S_2}$ we have
${\bigl|\log |w(P)|_v\bigr|\le B|G'|N\log R_v+B|G'|\log R_v}$.
Using~(\ref{ecalrest}), we obtain
\begin{equation}
\label{estwo}
\Xi_2\le 10dB|G'|N^2 +  dB|G'|(N^2+N)\calR\le dB|G'|N^2(\calR+11). 
\end{equation}

Futher, let~$\lambda$ be from item~(\ref{ilamb}) of Proposition~\ref{pw}. Then
${\height(\lambda)\le 12B|G'|N\log2\le 9B|G'|N}$,
because ${\height(1-\zeta)\le \log2}$ for a root of unity~$\zeta$. 
For ${v\in S_3}$ the number $j(P)$ is a $v$-adic integer. Hence so is the number $\lambda w(P)$. It follows that ${|w(P)|_v\le |\lambda^{-1}|_v}$ for ${v\in S_3}$, and 
\begin{equation}
\label{exithr}
\Xi_3\le \sum_{v\in S_3}d_v\log \left|\lambda^{-1}\right|_v\le d\height\left(\lambda^{-1}\right)=d\height(\lambda) \le 9dB|G'|N.
\end{equation}
Combining this with~(\ref{esss}),~(\ref{esone}) and~(\ref{estwo}), we obtain
${\height(P)\le B|G'|N^3(\calR+30)}$,
which is~(\ref{etbo}) with $|G|$ replaced by~$|G'|$. \qed

\section{Proof of Theorem~\ref{th1}}
\label{sth1}

It is similar and simpler than that of Theorem~\ref{tbo}. In this case ${d=1}$ and~$S$ consists of the 
infinite place of~$\Q$, whence 
${s=1}$, ${B=1}$  and  ${\calR=0}$.
We may take as the Runge unit~$w$ one of the functions ${w_\bfa^{\pm1}}$. We may 
assume that ${S_1=S}$ and ${S_2=\varnothing}$; otherwise we would have the estimate 
${\log |j(P)|\le N\log 50}$, which is much sharper than~(\ref{eth1}). We denote by~$c$ 
the nearest cusp to~$P$ with respect to the infinite place. 

Thus, ${\Xi_1=\log|w(P)|}$, and~(\ref{exioar}) now reads 
${\log\left|w(P)\right| \le  -N^{-1}\log|j(P)|+ \log|\beta(c)| + 2|G'|N}$.
Estimating ${\bigl|\log|\beta(c)|\bigr|}$ by ${12|G'|N\log N}$, we obtain 
${\Xi_1\le -N^{-1}\log|j(P)|+ 15|G'|N\log N}$.
Further, ${\Xi_2=0}$, and~$\Xi_3$ can be estimated by ${9|G'|N}$, according 
to~(\ref{exithr}). Since ${\Xi_1+\Xi_3=0}$, we obtain~(\ref{eth1}) with $|G|$ replaced 
by~$|G'|$. \qed

\section{A Special Case}
\label{sspec}
When passing from~(\ref{ewple}) to~(\ref{exioar}), we used the trivial estimate 
${\ord_cw\ge 1}$. One can improve our results by using a more elaborate lower 
bound for $\ord_cw$. In this section we apply this approach to the case when~$G$ is the 
normalizer of a split torus of prime level, which is important for the subsequent applications. 
Of course, similar strategy can be used in many different cases as well.

We start by recalling definitions and notations that will be in force for the rest of the article. 
Recall that a Cartan subgroup of the algebraic group $\GL_2$ over some ring  is a maximal 
subtorus, which can be either (totally) split or nonsplit. More precisely, fix a prime 
number~$p$. The subgroup of diagonal matrices in $\GL_{2} 
(\Z_{p})$ is a split Cartan subgroup, whose normalizer consists in diagonal and antidiagonal 
matrices. Given an integer ${n\geq 0}$, the normalizer of a split Cartan subgroup is the 
image in $\GL_{2} (\Z /p^n \Z )$ of a group conjugate to the above subgroup of $\GL_{2} 
(\Z_{p} )$. If~$G$ is such a group mod $p^n$, we denote by $X_\spl(p^n )$ the 
corresponding modular curve over $\Q$, and by~$Y_\spl(p^n )$ its finite part. 
On the other hand, let $\Z_{p^2}$ be the ring 
of integers of the unramified quadratic 
extension of $\Q_{p}$. By making the group $\Z_{p^2}^\times$ act on $\Z_{p^2}$ by 
multiplication, the choice of a $\Z_{p}$-basis of $\Z_{p^2}$ defines an embedding of 
$\Z_{p^2}^\times$ into $\GL_{2} (\Z_{p})$. The image of such an embedding is by 
definition a 
non split Cartan subgroup of $\GL_{2} (\Z_{p})$; it has index 2 in its normalizer. For $n$ 
any positive integer, reduction mod $p^n$ similarly defines (normalizer of) nonsplit Cartan 
subgroups in $\GL_{2} (\Z /p^n\Z )$. Those subgroups define in the usual
way modular curves over $\Q$, which we denote by $X_{\mathrm{nonsplit}} (p^n )$.

Now we focus on the case where ${N=p}$ is an odd prime number, and we let~$G$ be the 
normalizer of a split Cartan subgroup of $\GL_2(\F_p)$.  

\begin{theorem}
\label{tspto}
For any ${P\in Y_\spl(p)(\Z)}$ we have 
${\log|j(P)|\le 23p\log p}$.
\end{theorem}
(Note that Theorem~\ref{th1} with ${N=p}$  gives the bound 
${\log|j(P)|\le60p^2(p-1)^2\log p}$.)

\medskip

The curve~$X_\spl(p)$ has ${(p+1)/2}$ cusps, among which one (the cusp at 
infinity~$c_\infty$) is defined over~$\Q$ and the other ${(p-1)/2}$ are conjugate 
over~$\Q$, so that we have exactly~$2$ Galois orbits of cusps over~$\Q$. The covering 
${X(p)\to X_\spl(p)}$ is unramified at the cusps, and the covering ${X_\spl(p)\to 
X(1)}$ has ramification~$p$ at every cusp. 

We shall use the following lemma.

\begin{lemma}
\label{lord}
Let~$G$ be the normalizer of a split Cartan subgroup of $\GL_2(\F_p)$ and~$c$ a cusp of 
${X_G=X_\spl(p)}$. Then for any ${\bfa \in \bfA}$ we have 
\begin{equation}
\label{elord}
\left|\ord_cw_\bfa\right|\ge 2p(p-1)^2=p|G|.
\end{equation}
\end{lemma}

\begin{proof}
The proof relies on the identity
${\sum_{k=1}^{N-1}B_2(k/N)
=-(N-1)/6N}$,
where ${B_2(T)}$ is the second Bernoulli polynomial. 
We may assume that~$G$ is the normalizer of the diagonal subgroup of $\GL_2(\F_p)$. Also, replacing~$\bfa$ by $\bfa \sigma$ with a suitable ${\sigma\in \GL_2(\F_p)}$, we may assume that ${c=c_\infty}$. Since the covering ${X(p)\to X_G}$ is unramified at the cusps, we have, according to item~(\ref{iord}) of Proposition~\ref{pua},
${\ord_{c_\infty}w_\bfa= \sum_{\sigma\in G}\ord_{c_\infty}u_{\bfa \sigma}= 12p^2\sum_{\sigma\in G}\ell_{\bfa \sigma}}$.
Now we have two cases. If the entries of ${\bfa=(a_1,a_2)}$ are non-zero, then every non-zero element of~$\F_p$ occurs exactly ${2(p-1)}$ times as the first coordinate of $\bfa \sigma$, when~$\sigma$ runs over~$G$ (and~$0$ does not occur at all). Hence, by the definition of~$\ell_\bfa$, we have 
$$
\ord_{c_\infty}w_\bfa=  12p^2\cdot 2(p-1)\sum_{k=1}^{p-1} \frac12 B_2\left(\frac kp\right)= -2p(p-1)^2 =-|G|p,
$$
because  ${|G|=2(p-1)^2}$. And if either~$a_1$ or~$a_2$ is~$0$, then each non-zero element occurs exactly ${p-1}$ times, while~$0$ occurs ${(p-1)^2}$ times. Hence 
$$
\ord_{c_\infty}w_\bfa=  12p^2\left((p-1)\sum_{k=1}^{p-1} \frac12 B_2\left(\frac kp\right)+ (p-1)^2\cdot\frac12 B_2(0)\right)= p(p-1)^3=\frac12|G|p(p-1).
$$
Since ${p\ge 3}$, we have~(\ref{elord}) in any case. \qed
\end{proof}

\paragraph{Proof of Theorem~\ref{tspto}}
We argue as in the proof of Theorem~\ref{th1}. In particular, 
we again have~$w$ of the form $w_\bfa^{\pm1}$. Further,  ${\Xi_1=\log|w(P)|}$ and ${\Xi_2=0}$. But now, instead of~(\ref{exioar}) we  use~(\ref{ewple}), which gives
$$
\log\left|w(P)\right| \le  -\frac{\ord_cw}p\log|j(P)|+ \log|\beta(c)| + 2|G|p \le -\frac{\ord_cw}p\log|j(P)|+14|G|p\log p
$$
(we estimate ${\bigl|\log|\beta(c)|\bigr|}$ by ${12|G|p\log p}$). 
Hence ${\Xi_1\le |G|\left(-\log|j(P)|+ 14p\log p\right)}$ due to Lem\-ma~\ref{lord}.
We again use~(\ref{exithr}), which gives ${\Xi_3\le9|G|p}$. Since ${\Xi_1+\Xi_3=0}$, the result follows. \qed

\section{Split Cartan Structures in the Torsion of Elliptic Curves}
\label{sspli}

In this section we prove Theorems~\ref{tsp} and~\ref{ttwo}. 
We need several auxiliary results.  First of all, recall the results of Masser, W\"ustholz and  
Pellarin on the isogenies of elliptic curves.  Masser and W\"ustholz obtained an explicit 
upper bound for the degree of the minimal isogeny between two isogenous elliptic 
curves~\cite{MW90} and abelian varieties~\cite{MW93a}; see also~\cite{MW93}. 
Pellarin~\cite{Pe01} obtained a totally explicit version of this result for the case of elliptic 
curves. We use the result of Pellarin in the following form, which is a direct combination of 
Th\'eor\`eme~2 from~\cite{Pe01} and inequality~(51) on page~240 of~\cite{Pe01}.

\begin{proposition}
\label{ppel}
{\bf (Masser-W\"ustholz, Pellarin).} Let~$E$ be an elliptic curve defined over a number 
field~$K$ of degree~$d$. Let~$E'$ be another elliptic curve, defined over~$K$ and isogenous
to~$E$. Then there exists an isogeny  ${\psi:E\to E'}$ of degree
\begin{equation}
\label{edegis}
\deg\psi\le 10^{82}d^4\max\{1,\log d\}^2\left(1+\height(j_E)\right)^2.
\end{equation}
\end{proposition}

\begin{corollary}
\label{cpel}
Let~$E$ be a non-CM elliptic curve  defined over a number field~$K$ of degree~$d$, and 
admitting a cyclic isogeny over~$K$. Then the degree of this isogeny is bounded by the 
right-hand side of~(\ref{edegis}).
\end{corollary}
\begin{proof}
Let $\phi$ be a cyclic isogeny from $E$ to $E'$, and let $\phi^D \colon E' \to E$ be the 
dual isogeny. Let $\psi \colon E\to E'$ be a isogeny of degree bounded by the right-hand 
side of~(\ref{edegis}) which, without loss of generality, may be assumed to be cyclic. As
$E$ has no CM, the composed map $\phi^D \circ \psi$ must be multiplication by some 
integer, so that $\phi =\pm \psi$. \qed
\end{proof}

\medskip

Next we establish several simple properties of twists of elliptic curves.  
\begin{lemma}
\label{ltwist}
Let~$E$ be an elliptic curve over~$\Q$ with $j$-invariant ${j_{E}\ne 0,1728}$. Then there 
is a twist~$E'$ of~$E$ over~$\Q$, such that if $\ell\geq 5$ is a prime number dividing the 
conductor of~$E'$  then ${\ord_\ell\left(j_E(j_E-1728)\right)\ne0}$.
\end{lemma}

\begin{proof}
Consider the Weierstrass equation 
\begin{equation}
\label{ewei}
y^2 +xy =x^3 -\frac{36}{j_{E} -1728} x -
\frac1{j_{E} -1728} .
\end{equation}
It is known to have discriminant ${j_{E}^2 /(j_{E}-1728)^3}$, and that the elliptic 
curve~$E'$ it defines over~$\Q$ has $j$-invariant equal to $j_{E}$. It follows that~$E'$ is a 
twist of~$E$ over~$\Q$ and as $j\neq 0,1728$, this twist is necessarily quadratic. For a prime 
${\ell\ge 5}$ such that ${\ord_\ell\left(j_E(j_E-1728)\right)=0}$, 
equation~(\ref{ewei}) defines a smooth model for~$E'$ over~$\Z_{\ell}$. Therefore the 
minimal Weierstrass equation for~$E'$ over~$\Z$ defines a scheme which is smooth 
over~$\Z_\ell$, which means that~$\ell$ does not divide the conductor of~$E'$.\qed
\end{proof}

\medskip

We now fix a prime power~$p^n$. Let~$G$ be a subgroup of ${\GL (E[p^n])\simeq
\GL_{2} (\Z /p^n \Z )}$. We say that an elliptic curve~$E$ defined over~$\Q$ is 
\textsl{endowed with a $G$-level structure} if the image of the natural Galois 
representation $\rho_{E,p^n}\colon \gal (\bar{\Q} /\Q )\to \GL (E[p^n])$ is conjugate 
to~$G$. 

\begin{lemma} 
\label{lglev}
Assume that~$G$ contains $\pm 1$. Let~$E$ be an elliptic curve over~$\Q$ with ${j_E\ne 0, 
1728}$,  endowed with a $G$-level structure. Then any twist of~$E$ over~$\Q$ is endowed 
with a $G$-level structure as well.
\end{lemma}
\begin{proof} 
If~$E'$ is the twist ${E\otimes \chi}$ of~$E$ by a character~$\chi$, and $\rho_{E}$ 
is the Galois representation associated to the $p$-adic Tate module of~$E$, then the similar 
object $\rho_{E' }$ for~$E'$ is the tensor product ${\rho_{E}\otimes \chi}$. 
Since~$\chi$ has values in ${\{\pm1\}\subset G}$, the curve~$E$ is endowed with a
$G$-level structure if and only if the same is true for $E'$.  \qed
\end{proof}

\medskip
The following proposition is instrumental in the proof of Theorem~\ref{tsp}.
\begin{proposition}
\label{pubo} 
There exists an absolute effective constants~$\kappa$ such that the following holds. 
Let~$E$ be a non-CM elliptic curve over~$\Q$, endowed with a structure of normalizer of 
split Cartan subgroup in level~$p^n$. Then  
\begin{equation}
\label{enogrh}
p^n\le \kappa \left(1+\height(j_E)\right)^2.
\end{equation}
Assuming~GRH, we also have  
\begin{equation}
\label{ewithgrh}
p^n\le \kappa \log (N_E )(\log\log (2N_E ))^6,
\end{equation}
where~$N_E$ is the conductor of~$E$.
\end{proposition}

\begin{proof}
By the assumption, ${\rho_{E,p^n}\left(\gal(\bar\Q/\Q)\right)}$ is contained in the 
normalizer~$G$ of a split torus ${G_0\le \GL\left(E[p^n]\right)}$. Let~$\chi$ 
be the quadratic character of $\gal(\bar\Q /\Q)$ defined by $G/G_0$, and let~$K$ be the 
corresponding number field (which is at most quadratic over~$\Q$). Then ${\rho_{E,p^n}
\left(\gal(\bar\Q/K)\right)}$ is contained in the split torus~$G_0$, which  implies 
that~$E$ admits a cyclic isogeny of degree~$p^n$ over~$K$ (and even two distinct cyclic 
isogenies). This implies~(\ref{enogrh}) by Corollary~\ref{cpel}.  

   Now let us assume GRH and prove~(\ref{ewithgrh}). We apply the argument of 
Halberstadt and Kraus~\cite{HK98}, which makes use of Serre's explicit version of the 
Chebotarev theorem~\cite{Se81}. Let~$E'$ be the twist of~$E$ by~$\chi$. The 
conductors of~$E'$ and~$E$ are equal by \cite[Th\'eor\`eme 1]{HK98}. For any prime 
number~$\ell$ not dividing $pN_E$, the traces~$a_\ell$ and~$a_\ell'$ of a Frobenius 
substitution ${\mathrm{Frob}}_\ell$ acting on the $p$-adic Tate modules of~$E$ and~$E'$ 
satisfy~$a_\ell=a_\ell'\;\chi(\ell)$. The curve $E$ being endowed with a $K$-rational 
isogeny of degree $p$, it follows from Mazur's theorem on rational isogenies 
\cite[Theorem 1]{Ma78} that, if ${p>163}$, we have ${K\ne\Q}$ and, consequently, 
${\chi\ne 1}$. (One can replace here $163$ by $37$, because, as Mazur indicates in the introduction of~\cite{Ma78}, all curves with rational isogenies of order exceeding $37$ have complex multiplication.) Since~$E$ has no complex multiplication, $E$ and $E'$ are not 
$\Q$-isogenous, so we have ${a_\ell\ne a_\ell'}$ for infinitely many~$\ell$. 
Th\'eor\`eme~21 from~\cite{Se81} implies that, assuming GRH, one finds such~$\ell$ 
satisfying 
${\ell\leq c (\log N_E )^2 
(\log \log 2N_E )^{12}}$.
Since ${a_\ell\ne a_\ell'}$, we have  ${a_\ell\neq 0}$ and ${\chi(\ell)=-1}$, which 
means that $\rho_{E,p^n}\left(\mathrm{Frob}_\ell\right)$ belongs to ${G\setminus 
G_0}$. Since all elements from ${G\setminus G_0}$ have trace~$0$, we obtain ${p^n\mid 
a_\ell}$. Now Hasse's bounds imply that ${p^n \le |a_\ell |\le 2\sqrt\ell}$, which
yields~(\ref{ewithgrh}). \qed
\end{proof}

\paragraph{Proof of Theorem~\ref{tsp}} Assume that $X_{\mathrm{split}} (p^n )(\Q )$ 
has a non-CM and non-cuspidal point~$P$. We want to show that, for sufficiently large~$p$, 
we have ${n\le 2}$, and even ${n\le 1}$ assuming GRH. 

\begin{sloppypar}
Our point~$P$ gives rise to a non-CM elliptic curve~$E$. Since~$P$ induces a point in 
$X_{\mathrm{split}} (p)(\Q )$,  the results of Momose and Merel 
\cite[Theorem 3.1]{Me05} imply that either $p\le13$ or ${j(P)=j_{E}}$ belongs to~$\Z$. 
Now Theorem~\ref{tspto} yields  
\begin{equation}
\label{elogj}
\log|j_E|\le 23p\log p,
\end{equation} 
which,  together with~(\ref{enogrh})   gives  ${p^n\le c(p\log p)^2}$ for some constant~$c$ (since ${j_E\in \Z}$ we have 
${\height(j_E)=\log|j_E|}$). Hence 
${n\le 2}$ for sufficiently large~$p$,  proving the first (unconditional) statement. 
\end{sloppypar}

Now let us prove the second statement. Lemmas~\ref{ltwist} and~\ref{lglev} allow us to 
assume (replacing~$E$, if necessary, by a quadratic twist) that every prime ${\ell\geq 5}$
dividing $N_{E}$, divides either~$j_{E}$ or ${j_{E}-1728}$. The curve~$E$ has potential 
good reduction at all primes, so ${\ord_\ell (N_{E})=2}$, and the exponents of 
the conductor at~$2$ and~$3$ are at most~$8$ and~$5$ respectively. Therefore 
${N_E\le 2^8\cdot3^5\cdot j_{E}^2(j_{E}-1728)^2}$. 
Combining this with~(\ref{ewithgrh}) and~(\ref{elogj}), we obtain, assuming GRH, that
${p^n\le cp(\log p)^7}$ for some constant~$c$. Therefore ${n\le 1}$ for sufficiently 
large~$p$. \qed

\paragraph{Proof of Theorem~\ref{ttwo}}
The proof is very similar to that of the first part of Theorem~\ref{tsp}. Let~$p$, $q$ 
and~$r$ be distinct split Cartan deficient primes for a non-CM elliptic curve $E/\Q$. We 
assume ${11\le p<q<r}$. Again applying the results of Momose and Merel, we obtain ${j_E\in 
\Z}$. Hence~$E$ gives rise to a point on  $Y_{\mathrm{split}}(p)(\Z)$, and 
Theorem~\ref{tspto} yields~(\ref{elogj}).

On the other hand, over some quadratic field~$E$ admits a cyclic isogeny of degree~$p$, and 
the same is true for~$q$ and~$r$. Hence over some field of degree (at most)~$8$ the 
curve~$E$ admits a cyclic isogeny of degree $pqr$. Using Corollary~\ref{cpel} 
and~(\ref{elogj}) we  obtain 
${p^3\le pqr\le c(p\log p)^2}$,
which is impossible when~$p$ exceeds certain~$p_0$. \qed

{\footnotesize

\end{document}